\documentclass[10pt]{amsart}

\usepackage[english]{babel}
\usepackage{amssymb}
\usepackage[leqno]{amsmath}
\usepackage{mathrsfs}
\usepackage{stmaryrd}
\usepackage{chemarrow}
\usepackage[norelsize]{algorithm2e}
\usepackage{enumerate}
\usepackage{graphicx}
\usepackage{subcaption}
\usepackage[all]{xy}
\usepackage{tikz}
\usetikzlibrary{arrows}
\usepackage{multirow}
\usepackage[colorlinks=true,allcolors=blue]{hyperref}
\usepackage{booktabs}
\usepackage{float}
\usepackage{placeins}

\newtheorem{theorem}{Theorem}[section]
\newtheorem{lemma}[theorem]{Lemma}
\newtheorem{corollary}[theorem]{Corollary}

\newtheorem{example}[theorem]{Example}

\newtheorem{remark}[theorem]{Remark}

\newcommand{\dd}{\,{\rm d}}

\newcommand{\curl}{\operatorname{curl}}
\renewcommand{\div}{\operatorname{div}}
\newcommand{\grad}{\operatorname{grad}}
\newcommand{\tr}{\operatorname{tr}}

\newcommand{\dev}{\operatorname{dev}}
\newcommand{\sym}{\operatorname{sym}}
\newcommand{\skw}{\operatorname{skw}}

\numberwithin{equation}{section}

\begin{document}

\title[Weakly Symmetric and Traceless Tensor Elements]
{Weakly Symmetric and Traceless Tangential--Normal Tensor Finite Elements:
Application to the Brinkman Equations}

\author{Xuehai Huang}
\address{School of Mathematics, Shanghai University of Finance and Economics,
Shanghai 200433, China}
\email{huang.xuehai@sufe.edu.cn}

\author{Xinyue Zhao}
\address{School of Mathematics, Shanghai University of Finance and Economics,
Shanghai 200433, China}
\email{zhaoxinyue20210921@163.com}

\thanks{The first author was supported by the National Natural Science
Foundation of China Project 12671432.}

\begin{abstract}
We develop a family of weakly symmetric and pointwise traceless tangential--normal tensor finite elements in arbitrary space dimension and for all polynomial orders. Symmetry is imposed through local cell moments, while the only globally coupled stress degrees of freedom are tangential--normal facet moments; no vertex degrees of freedom are required. In dimensions three and higher, a lowest-order linear enrichment restores the rigid-motion facet control required for discrete Korn stability. As a principal application, we construct a distributional mixed method for the incompressible Brinkman equations using the physical viscous stress. Coupled with divergence-conforming BDM velocities and discontinuous pressures, the method is stabilization-free, uniformly stable with respect to the viscosity parameter, exactly divergence-free, and pressure-robust. We establish optimal-order error estimates in the natural norms. Under suitable parameter-explicit regularity assumptions, we also obtain a parameter-uniform boundary-layer estimate with optimal Darcy approximation order. Relaxing tangential--normal continuity yields an algebraically equivalent stress-hybridized formulation and a stabilization-free virtual element realization.
\end{abstract}



\maketitle

\section{Introduction}

We develop weakly symmetric and pointwise traceless tangential--normal
tensor finite elements. The main contribution is the finite element family
itself: it preserves pointwise tracelessness while imposing symmetry weakly
through cell moments and requiring only tangential--normal facet coupling
between neighboring elements. As a principal application, we use these
elements to discretize the incompressible Brinkman equations with the
symmetric-gradient viscous operator.

Available conforming constructions for symmetric and traceless tensors in three dimensions involve high polynomial degrees and continuity constraints on lower-dimensional subsimplices. In particular, the
\(H(\div;\mathbb S\cap\mathbb T)\)-conforming element of
\cite{HuLinShi2023} has lowest polynomial degree seven and uses vertex
derivatives through order three. For
\(H(\div\div;\mathbb S\cap\mathbb T)\)-conformity, the construction of
\cite{Huang2025} admits polynomial degree four with tensor values at
vertices, whereas that of \cite{GuoHuLin2025} starts at polynomial degree
six and uses first-order vertex derivatives. These constructions motivate
relaxing the interelement continuity requirements and imposing symmetry only
weakly. However, they do not yield the weakly symmetric tensor space
developed here with facet-only interelement coupling.

Tangential--normal continuous traceless tensor elements were introduced in
\cite{GopalakrishnanLedererSchoeberl2020}. The same trace structure was
subsequently incorporated into the three-dimensional distributional finite
element \(\curl\div\) complex for traceless tensors in
\cite{ChenHuangZhang2023}. Our earlier Stokes work
\cite{ChenHuangZhangZhao2026} employed the same tangential--normal trace
structure for a generally nonsymmetric traceless pseudostress. Unlike
\cite{ChenHuangZhangZhao2026}, however, 
the present formulation approximates the physical viscous stress \(\nu\boldsymbol{\varepsilon}(\boldsymbol u)\), which is symmetric. This change requires weak symmetry, discrete Korn control, and a new lowest-order treatment. Weak symmetry has also been incorporated into tangential--normal stress
discretizations. The weakly symmetric mass-conserving mixed stress (MCS)
method of
\cite{Gopalakrishnan2020} directly approximates the viscous stress, imposes
symmetry through an independent vorticity variable, and enriches the stress
space with matrix bubbles to obtain discrete stability. At lowest order in
three dimensions, \cite{GopalakrishnanKoglerLedererSchoeberl2023} develops
related mixed and HDG formulations with minimal facet coupling, whereas the
present \({\rm RM}(F)\)-enrichment provides the required rotational trace
control directly through the stress space, without the
vorticity--divergence stabilization used there. These works already provide
important properties such as exact mass conservation and pressure robustness.
Accordingly, the distinguishing feature of the present
construction is that weak symmetry is incorporated directly into a
 pointwise traceless local tensor space through cell-moment constraints,
 without an independent vorticity unknown. The resulting family is valid
 for every \(d\ge2\) and \(k\ge0\).

For positive polynomial degree, the element is determined by
tangential--normal facet moments \eqref{eq:tndofTS1} and symmetric-traceless cell moments \eqref{eq:tndofTS2}. At
lowest order in dimensions three and higher, the constant facet moments fail
to detect the rotational traces of elementwise rigid motions. We therefore
enrich the local stress space by linear modes whose tangential--normal traces
recover these missing rotations. This enrichment restores the facet control
required by the projected discrete Korn inequality without increasing the
approximation order; no enrichment is needed in two dimensions.

We apply the new elements to the incompressible Brinkman problem with the
symmetric-gradient viscous operator. On a bounded polytope
\(\Omega\subset\mathbb R^d\) with \(d\ge2\), we consider
\begin{equation}\label{eq:Brinkman_intro}
	\begin{cases}
		-\div\bigl(\nu\boldsymbol{\varepsilon}(\boldsymbol u)\bigr)
		+\boldsymbol u-\nabla p=\boldsymbol f,
		&\text{in }\Omega,\\
		\div\boldsymbol u=0,
		&\text{in }\Omega,\\
		\boldsymbol u=\boldsymbol0,
		&\text{on }\partial\Omega,
	\end{cases}
\end{equation}
where \(0<\nu\le1\) and
\(\boldsymbol{\varepsilon}(\boldsymbol u)
=(\nabla\boldsymbol u+\nabla\boldsymbol u^{\intercal})/2\).
Introducing the physical viscous stress
\(\boldsymbol\sigma:=\nu\boldsymbol{\varepsilon}(\boldsymbol u)\), we rewrite
\eqref{eq:Brinkman_intro} as
\begin{equation}\label{eq:Brinkman1st_intro}
	\boldsymbol\sigma=\nu\boldsymbol{\varepsilon}(\boldsymbol u),\quad
	-\div\boldsymbol\sigma+\boldsymbol u-\nabla p=\boldsymbol f,\quad
	\div\boldsymbol u=0\quad\text{in }\Omega;\quad
	\boldsymbol u=\boldsymbol0\quad\text{on }\partial\Omega.
\end{equation}
Thus \(\boldsymbol\sigma\) has precisely the symmetric--traceless structure
targeted by the elements constructed above.

As \(\nu\to0\), \eqref{eq:Brinkman_intro} formally approaches the Darcy limit,
whose velocity satisfies only a normal boundary condition; the no-slip
condition may therefore generate boundary layers
\cite{MardalTaiWinther2002,HuangWang2023}. An \(H(\div)\)-conforming
velocity space is compatible with the limiting boundary condition and,
when paired with a compatible pressure space, exact mass conservation.

Uniformly stable finite element methods for Brinkman and Darcy--Stokes models were studied, for example, in \cite{XieXuXue2008,JuntunenStenberg2010,HongKraus2016}. Representative nonconforming
and \(H(\div)\)-conforming methods include
\cite{GuzmanNeilanBrinkman2012,KonnoStenberg2011}, while
\(H(\div)\)-conforming HDG and weak Galerkin discretizations were developed in
\cite{FuJinQiu2019,Mu2020Brinkman}. Divergence-free conforming virtual
element methods, including parameter-robust analysis in the Darcy limit,
were studied in \cite{HuangWang2023}.

Tensor formulations are distinguished by the tensor variable being
approximated. Nonsymmetric pseudostress formulations for the Brinkman
problem include mixed, virtual element, and weak Galerkin methods
\cite{GaticaGaticaMarquez2014,CaceresGaticaSequeira2017,Gharibi2024}.
The three-field method of \cite{HowellNeilanWalkington2016} introduces
the deviatoric part of the velocity gradient together with a total stress,
whereas strongly symmetric formulations approximate a pseudostress
containing the pressure or the Cauchy stress
\cite{QianWuWang2020,MeddahiRuizBaier2022}. In contrast, the tensor
variable considered here is the physical viscous stress
\(\nu\boldsymbol{\varepsilon}(\boldsymbol u)\), which contains no pressure term and is
pointwise traceless. In the discrete construction, tracelessness is imposed
pointwise, whereas symmetry is enforced weakly through cell moments.

Building on the distributional \(\curl\div\) framework for traceless
tensors, we establish the stable decomposition
\begin{equation}\label{eq:introStressDecomposition}
	(H_0(\div,\Omega))'
	=
	\div H^{-1}(\curl\div,\Omega;\mathbb S\cap\mathbb T)
	+\nabla L_0^2(\Omega).
\end{equation}
This decomposition provides the key continuous inf--sup ingredient for the
distributional mixed formulation \eqref{eq:BrinkmanMixed} and leads to
parameter-uniform well-posedness. We prove that this formulation is
equivalent to the primal Brinkman problem. Its discretization couples the
new stress elements with
\(\mathrm{BDM}_{k+1}\) velocities
\cite{BrezziDouglasMarini1985,Nedelec1986,
	BrezziDouglasDuranFortin1987}
and discontinuous \(\mathbb P_k\) pressures. Uniform discrete stability
relies on the projected discrete Korn estimate on the divergence-free
velocity kernel together with the BDM--pressure inf--sup condition. The
resulting distributional method is stabilization-free, uniformly well posed in
\(\nu\)-fitted norms, exactly divergence-free, and pressure-robust.
We establish optimal-order error estimates in the natural norms for smooth
solutions and a parameter-uniform boundary-layer estimate by comparison
with the Darcy limit.


Finally, relaxing tangential--normal continuity yields an algebraically equivalent
stress-hybridized formulation and a stabilization-free virtual element
realization on simplices, while omitting the drag term gives the
corresponding Stokes method.

The remainder of the paper is organized as follows. Section~\ref{sec:weaklysymmtrictracelesstnelements} constructs
the tensor elements. Sections~\ref{sec:Brinkman} and~\ref{sec:BrinkmanDiscrete} present the continuous distributional
formulation and the mixed finite element method, respectively. Section~\ref{sec:error}
establishes the error estimates. Section~\ref{sub:equivformulations} derives the equivalent
stress-hybridized and virtual element formulations, and Section~\ref{sec:BrinkmanNumericalResults} presents
the numerical experiments.
\section{Weakly Symmetric and Traceless Tangential-Normal Tensor Elements}\label{sec:weaklysymmtrictracelesstnelements}

This section constructs tangential--normal finite elements for pointwise traceless tensors, with symmetry imposed weakly through cell moments against skew-symmetric tensors. The resulting elements are valid in all space dimensions \(d\ge2\) and involve only facet and cell degrees of freedom.

\subsection{Notation}
Let \(\mathbb M:=\mathbb R^{d\times d}\), and denote by \(\mathbb S\),
\(\mathbb K\), and \(\mathbb T\) its subspaces of symmetric,
skew-symmetric, and traceless matrices, respectively. Thus
\(\mathbb S\cap\mathbb T\) is the space of symmetric and traceless matrices.
For \(\boldsymbol\tau\in\mathbb M\), set
\[
\sym\boldsymbol\tau:=\frac{1}{2}(\boldsymbol\tau+\boldsymbol\tau^{\intercal}),
\quad
\skw\boldsymbol\tau:=\frac{1}{2}(\boldsymbol\tau-\boldsymbol\tau^{\intercal}),
\quad
\dev\boldsymbol\tau:=\boldsymbol\tau-\frac{1}{d}(\tr\boldsymbol\tau)\boldsymbol I.
\]

Let \(\Omega\subset\mathbb R^d\), \(d\ge2\), be a bounded connected polytope. For a
bounded Lipschitz domain \(D\subset\mathbb R^d\) and an integer \(m\ge0\),
\(H^m(D)\) denotes the standard Sobolev space with norm
\(\|\cdot\|_{m,D}\) and seminorm \(|\cdot|_{m,D}\). Let \(H_0^m(D)\) be the
closure of \(C_0^\infty(D)\) in \(H^m(D)\). We set \(L^2(D)=H^0(D)\), with
inner product \((\cdot,\cdot)_D\) and norm \(\|\cdot\|_{0,D}\). When
\(D=\Omega\), the subscript \(D\) is omitted. For any \(D\), \(h_D\) denotes
its diameter and \(\boldsymbol n_{\partial D}\) its unit outward normal; we
write simply \(\boldsymbol n\) when the domain is clear.

We use
\[
\begin{aligned}
	H(\div,D)
	&:=\{\boldsymbol v\in L^2(D;\mathbb R^d):
	\div\boldsymbol v\in L^2(D)\},\\
	H_0(\div,D)
	&:=\{\boldsymbol v\in H(\div,D):
	\boldsymbol v\cdot\boldsymbol n=0\text{ on }\partial D\},
\end{aligned}
\]
with norm
\(
\|\boldsymbol v\|_{H(\div,D)}
:=(\|\boldsymbol v\|_{0,D}^2+\|\div\boldsymbol v\|_{0,D}^2)^{1/2}.
\)
The boundary condition in \(H_0(\div,D)\) is understood in the normal-trace
sense. Let \(L_0^2(D)\) denote the space of square-integrable functions with vanishing mean value.
The duality pairing
between a space and its dual is denoted by
\(\langle\cdot,\cdot\rangle\).
For an integer \(m\ge0\), let
\(\mathbb P_m(D)\) be the space of polynomials on \(D\) of total degree at
most \(m\), with the convention \(\mathbb P_{-1}(D):=\{0\}\).

Let \(\{\mathcal T_h\}_{h>0}\) be a shape-regular family of simplicial meshes
of \(\Omega\), with \(h=\max_{T\in\mathcal T_h}h_T\) and
\(h_T=\operatorname{diam}(T)\). We denote by \(\mathcal F_h\),
\(\mathring{\mathcal F}_h\), and \(\mathcal F_h^\partial\) the sets of all
facets, interior facets, and boundary facets, respectively, and set
\(h_F:=\operatorname{diam}(F)\). For \(T\in\mathcal T_h\), let
\(
\mathcal F(T):=\{F\in\mathcal F_h:F\subset\partial T\}
\)
be its set of facets, while \(\partial T\) denotes its geometric boundary.

For later use in the local finite element construction, let
\(\lambda_0,\ldots,\lambda_d\) be the barycentric coordinates of a simplex
\(T\in\mathcal T_h\) with vertices
\(\texttt{v}_0,\ldots,\texttt{v}_d\). For \(j\ne \ell\), set
\(\boldsymbol t_{j,\ell}:=\texttt{v}_\ell-\texttt{v}_j\). We denote by
\(F_i\) the facet opposite to \(\texttt{v}_i\), and by \(\boldsymbol n_i\)
the unit outward normal to \(F_i\). On each facet \(F\), fix an orthonormal
basis \(\{\boldsymbol t_{F,i}\}_{i=1}^{d-1}\) of its tangent space.

For each interior facet \(F=T^+\cap T^-\), we fix a unit normal
\(\boldsymbol n_F\) to orient jumps, with \(\boldsymbol n_F\) taken outward on
\(T^+\). For boundary facets \(F\subset\partial\Omega\), we set
\(\boldsymbol n_F=\boldsymbol n_{\partial\Omega}|_F\). For a piecewise field
\(w\), its jump across \(F\) is defined by
\[
[\![w]\!]:=
\begin{cases}
	w^+-w^- , & F=T^+\cap T^-\in\mathring{\mathcal F}_h,\\
	w, & F\in\mathcal F_h^\partial .
\end{cases}
\]
When an element \(T\) is fixed, \(\boldsymbol n\) denotes its outward unit
normal on \(\partial T\).

For a vector field \(\boldsymbol v\), we write
\[
\boldsymbol{\varepsilon}(\boldsymbol v):=\sym\grad\boldsymbol v,
\qquad
\dev\boldsymbol{\varepsilon}(\boldsymbol v):=\dev\bigl(\boldsymbol{\varepsilon}(\boldsymbol v)\bigr).
\]
The divergence of a matrix field is taken row-wise. We denote the broken
gradient and divergence by \(\grad_h\) and \(\div_h\), respectively, and set
\(\boldsymbol{\varepsilon}_h(\boldsymbol v):=\sym\grad_h\boldsymbol v\).

For a vector field \(\boldsymbol w\), the tangential projection on a facet \(F\)
is
\[
\Pi_F\boldsymbol w
:=
\bigl.(\boldsymbol I-\boldsymbol n_F\boldsymbol n_F^{\intercal})
\boldsymbol w\bigr|_F .
\]
For a matrix field \(\boldsymbol\tau\), we write
\(\Pi_F\boldsymbol\tau\boldsymbol n_F\) for
\(\Pi_F(\boldsymbol\tau\boldsymbol n_F)\), the tangential component of the
normal trace. Tangential boundary constraints are understood facetwise on all
\(F\in\mathcal F_h^\partial\). Whenever \(\mathbb R^{d-1}\)-valued functions
are defined on a facet, we identify \(\mathbb R^{d-1}\) with the tangent space
of that facet through the fixed tangential basis.

For a linear space \(V(D)\), define the broken space
\(
V(\mathcal T_h)
:=\prod_{T\in\mathcal T_h}V(T).
\)
For a finite-dimensional vector or matrix space \(\mathbb X\), set
\[
V(D;\mathbb X):=V(D)\otimes\mathbb X,
\qquad
V(\mathcal T_h;\mathbb X):=\prod_{T\in\mathcal T_h}V(T;\mathbb X).
\]
The \(L^2\)-orthogonal projections onto \(\mathbb P_m(D;\mathbb X)\) and
\(\mathbb P_m(\mathcal T_h;\mathbb X)\) are denoted by \(Q_{m,D}\) and
\(Q_{m,h}\), respectively, with the target space understood from context.
In particular, when applied to tensor fields,
\(
Q_{k,h}:L^2(\Omega;\mathbb M)\longrightarrow
\mathbb P_k(\mathcal T_h;\mathbb M)
\)
preserves the subspaces \(\mathbb T\) and \(\mathbb K\). The same notation
will be used for the scalar projector onto the discrete pressure space.

The notation \(a\lesssim b\) means \(a\le Cb\), with a generic constant
\(C\) independent of parameter \(\nu\) and mesh size \(h\). We write \(a\eqsim b\) if both \(a\lesssim b\) and
\(b\lesssim a\) hold.

\subsection{An auxiliary traceless tangential-normal element}

For each \(\ell=0,\ldots,d\), choose
\(j_\ell\in\{0,\ldots,d\}\setminus\{\ell\}\) and set
\(
I_\ell:=\{0,\ldots,d\}\setminus\{\ell,j_\ell\}.
\)

We first recall the local tangential-normal bubble space for traceless tensors. For
\(\boldsymbol\tau\in\mathbb P_k(T;\mathbb T)\), define its
tangential-normal trace by
\(
\tr_{T,F}^{tn}\boldsymbol\tau
:=
\Pi_F(\boldsymbol\tau\boldsymbol n_{\partial T}).
\)
When the element \(T\) is fixed, we simply write \(\tr_F^{tn}\).
For \(k\ge1\), let the bubble space
\[
\mathbb B_k^{tn}(T;\mathbb T)
:=
\{\boldsymbol\tau\in\mathbb P_k(T;\mathbb T):
\tr_F^{tn}\boldsymbol\tau=0
\quad\forall\,F\in\mathcal F(T)\},
\]
and set \(\mathbb B_0^{tn}(T;\mathbb T):=\{0\}\). By
\cite[Lemma~3.3 and Remark~3.4]{ChenHuangZhang2023}, for \(k\ge1\),
\begin{equation}\label{eq:bubble}
	\mathbb B_k^{tn}(T;\mathbb T)
	=
	\mathbb P_{k-1}(T)\otimes
	\operatorname{span}
	\left\{
\lambda_\ell\,\dev(\boldsymbol n_i\otimes\boldsymbol t_{i,\ell})
	:\ \ell=0,\ldots,d,\ i\in I_\ell
	\right\}.
\end{equation}

For \(k\ge0\), define
\[
\Sigma_k^+(T;\mathbb T)
:=
\mathbb P_k(T;\mathbb T)
+\mathbb B_{k+1}^{tn}(T;\mathbb T).
\]
Here and below, \(\mathbb P_k(F_\ell)\) is identified with its canonical
degree-\(k\) homogeneous extension in the barycentric coordinates
\(\{\lambda_i\}_{i\ne\ell}\) to \(T\).
By the geometric decomposition in \cite[Lemma~3.3 and Remark~3.4]{ChenHuangZhang2023},
	\begin{equation}\label{eq:geometric}
		\mathbb P_k(T;\mathbb T)
		=
		\bigoplus_{\ell=0}^{d}
		\left(
		\mathbb P_k(F_\ell)\otimes
		\operatorname{span}
		\left\{
\dev(\boldsymbol n_i\otimes\boldsymbol t_{i,\ell})
		\right\}_{i\in I_\ell}
		\right)
		\oplus
		\mathbb B_k^{tn}(T;\mathbb T).
	\end{equation}
	Since
	\(\mathbb P_k(T;\mathbb T)\cap\mathbb B_{k+1}^{tn}(T;\mathbb T)
	=\mathbb B_k^{tn}(T;\mathbb T)\), combining this decomposition with the
	definition of \(\Sigma_k^+(T;\mathbb T)\), we obtain
	\begin{equation}\label{eq:geometricSigma}
		\Sigma_k^+(T;\mathbb T)
		=
		\bigoplus_{\ell=0}^{d}
		\left(
		\mathbb P_k(F_\ell)\otimes
		\operatorname{span}
		\left\{
\dev(\boldsymbol n_i\otimes\boldsymbol t_{i,\ell})
		\right\}_{i\in I_\ell}
		\right)
		\oplus
		\mathbb B_{k+1}^{tn}(T;\mathbb T).
	\end{equation}
The degrees of freedom for
\(\Sigma_k^+(T;\mathbb T)\) are
\begin{subequations}\label{eq:tndof}
	\begin{align}
		(\boldsymbol t_{F,i}^{\intercal}\boldsymbol\tau\boldsymbol n, q)_F,
		&\qquad
		q\in\mathbb P_k(F),\;
		i=1,\ldots,d-1,\; F\in\mathcal F(T),
		\label{eq:tndof1}\\
		(\boldsymbol\tau,\boldsymbol q)_T,
		&\qquad
		\boldsymbol q\in\mathbb P_k(T;\mathbb T).
		\label{eq:tndof2}
	\end{align}
\end{subequations}

\begin{lemma}\label{le:unisolventT}
	The DoFs \eqref{eq:tndof} are unisolvent for
	\(\Sigma_k^+(T;\mathbb T)\).
\end{lemma}

\begin{proof}
By \eqref{eq:geometricSigma} and \(\dim\mathbb T=d^2-1\),
the number of DoFs in \eqref{eq:tndof} agrees with
\(\dim\Sigma_k^+(T;\mathbb T)\).

	It remains to prove uniqueness. Let
	\(\boldsymbol\tau\in\Sigma_k^+(T;\mathbb T)\) have all the DoFs in
	\eqref{eq:tndof} equal to zero.
	Since the components of \(\tr_F^{tn}\boldsymbol\tau\) in the chosen
	tangent basis belong to \(\mathbb P_k(F)\), the vanishing facet moments
	imply \(\tr_F^{tn}\boldsymbol\tau=0\) for all \(F\in\mathcal F(T)\). Thus
	\(\boldsymbol\tau\in\mathbb B_{k+1}^{tn}(T;\mathbb T)\). Set
\(A_{\ell i}:=\dev(\boldsymbol n_i\otimes\boldsymbol t_{i,\ell})\).
	By the \(k=0\) case of the geometric decomposition \eqref{eq:geometric},
	\(\{A_{\ell i}:\ell=0,\ldots,d,\ i\in I_\ell\}\) forms a basis of
	\(\mathbb T\). Let \(\{A_{\ell i}^*\}\) be its Frobenius-dual basis. By
	\eqref{eq:bubble},
	\[
	\boldsymbol\tau
	=
	\sum_{\ell=0}^{d}\sum_{i\in I_\ell}
	\lambda_\ell p_{\ell i}A_{\ell i},
	\qquad p_{\ell i}\in\mathbb P_k(T).
	\]
	Taking \(\boldsymbol q=\sum_{\ell=0}^{d}\sum_{i\in I_\ell}
	p_{\ell i}A_{\ell i}^*\) in \eqref{eq:tndof2} gives
	\[
	0=(\boldsymbol\tau,\boldsymbol q)_T
	=
	\sum_{\ell=0}^{d}\sum_{i\in I_\ell}
	\int_T\lambda_\ell\lvert p_{\ell i}\rvert^2\,\mathrm dx.
	\]
	Since \(\lambda_\ell>0\) in the interior of \(T\), we have
	\(p_{\ell i}=0\) for all \(\ell\) and \(i\), and thus
	\(\boldsymbol\tau=0\).
\end{proof}

\subsection{Weak symmetry through cell moments for \texorpdfstring{\(k\ge1\)}{k >= 1}}

For \(k\ge1\), define
\begin{equation}\label{eq:Sigma+TS}
\Sigma_k^+(T;\mathbb S\cap\mathbb T)
:=
\left\{
\boldsymbol\tau\in\Sigma_k^+(T;\mathbb T):
(\boldsymbol\tau,\boldsymbol q)_T=0
\quad
\forall\,\boldsymbol q\in\mathbb P_k(T;\mathbb K)
\right\}.
\end{equation}
Although the notation includes \(\mathbb S\), functions in
\(\Sigma_k^+(T;\mathbb S\cap\mathbb T)\) are generally not pointwise
symmetric; symmetry is imposed only through the above cell-moment conditions.
Since
\(\Sigma_k^+(T;\mathbb T)\subset\mathbb P_{k+1}(T;\mathbb T)\), set
\[
\mathbb P_{k+1}^{\perp k}(T;\mathbb K)
:=
\left\{
\boldsymbol r\in\mathbb P_{k+1}(T;\mathbb K):
(\boldsymbol r,\boldsymbol q)_T=0
\quad
\forall\,\boldsymbol q\in\mathbb P_k(T;\mathbb K)
\right\}.
\]
The \(L^2(T)\)-orthogonality of the symmetric and skew-symmetric parts
shows that \eqref{eq:Sigma+TS} is equivalently characterized by
\[
\Sigma_k^+(T;\mathbb S\cap\mathbb T)
=
\left\{
\boldsymbol\tau\in\Sigma_k^+(T;\mathbb T):
\skw\boldsymbol\tau\in
\mathbb P_{k+1}^{\perp k}(T;\mathbb K)
\right\}.
\]
Thus weak symmetry annihilates all skew-symmetric moments of degree at most
\(k\). Unlike the trace-based decomposition \eqref{eq:geometricSigma}, this
is an intrinsic characterization: the cell-moment constraint couples the
facet and bubble modes.
The DoFs are
\begin{subequations}\label{eq:tndofTS}
	\begin{align}
		(\boldsymbol t_{F,i}^{\intercal}\boldsymbol\tau\boldsymbol n, q)_F,
		&\qquad
		q\in\mathbb P_k(F),\;
		i=1,\ldots,d-1,\; F\in\mathcal F(T),
		\label{eq:tndofTS1}\\
		(\boldsymbol\tau,\boldsymbol q)_T,
		&\qquad
		\boldsymbol q\in
		\mathbb P_k(T;\mathbb S\cap\mathbb T).
		\label{eq:tndofTS2}
	\end{align}
\end{subequations}

\begin{lemma}\label{le:unisolventTS}
	For \(k\ge1\), the DoFs \eqref{eq:tndofTS} are unisolvent for
	\(\Sigma_k^+(T;\mathbb S\cap\mathbb T)\).
\end{lemma}

\begin{proof}
	Since \(\mathbb P_k(T;\mathbb K)\subset\Sigma_k^+(T;\mathbb T)\)
	and the \(L^2(T)\)-pairing is nondegenerate, the weak-symmetry conditions
	in \eqref{eq:Sigma+TS} are independent. Consequently,
	\[
	\begin{aligned}
	\dim\Sigma_k^+(T;\mathbb S\cap\mathbb T)
	&=\dim\Sigma_k^+(T;\mathbb T)-\dim\mathbb P_k(T;\mathbb K)\\
	&=(d+1)(d-1)\binom{k+d-1}{d-1}+\frac{1}{2}(d-1)(d+2)\binom{k+d}{d}.
	\end{aligned}
	\]
	This agrees with the number of DoFs in \eqref{eq:tndofTS}. It remains to prove uniqueness. Let
	\(\boldsymbol\tau\in\Sigma_k^+(T;\mathbb S\cap\mathbb T)\) have all DoFs in \eqref{eq:tndofTS}
	equal to zero. By the definition of
	\(\Sigma_k^+(T;\mathbb S\cap\mathbb T)\),
	\[
	(\boldsymbol\tau,\boldsymbol q)_T=0
	\qquad
	\forall\,\boldsymbol q\in\mathbb P_k(T;\mathbb K).
	\]
	Since
	\(
	\mathbb P_k(T;\mathbb T)
	=
	\mathbb P_k(T;\mathbb S\cap\mathbb T)
	\oplus
	\mathbb P_k(T;\mathbb K),
	\)
	all cell moments against \(\mathbb P_k(T;\mathbb T)\) vanish. Together with
	the vanishing facet moments, Lemma~\ref{le:unisolventT} gives
	\(\boldsymbol\tau=0\).
\end{proof}

\subsection{Lowest-order element and \texorpdfstring{\({\rm RM}(F)\)}{RM(F)} facet moments}

We now treat the lowest-order case \(k=0\) uniformly for all dimensions
\(d\ge2\). The facet coupling must detect the tangential traces of
elementwise rigid motions.
Let
\[
{\rm RM}(T):=\mathbb P_0(T;\mathbb R^d)\oplus \mathbb K\boldsymbol x
\]
be the space of elementwise rigid motions. On a facet \(F\), define
\[
{\rm RM}(F):=\mathbb P_0(F;\mathbb R^{d-1})
\oplus \mathbb K_F\boldsymbol x_F ,
\]
where \(\boldsymbol x_F:=\Pi_F(\boldsymbol x-\boldsymbol x_F^c)\),
\(\boldsymbol x_F^c:=|F|^{-1}\int_F\boldsymbol x\,ds\), and
\(
\mathbb K_F:=\{\boldsymbol A\in\mathbb K:\boldsymbol A\boldsymbol n_F=0\}
\)
is the space of skew-symmetric matrices acting on the tangent space of \(F\). Here
\(\mathbb K_F\boldsymbol x_F
:=\{\boldsymbol A\boldsymbol x_F:\boldsymbol A\in\mathbb K_F\}\).
Since \(\dim\mathbb K_F=(d-1)(d-2)/2\), the rotational component has one
mode per facet when \(d=3\). When \(d=2\), \(\mathbb K_F=\{\boldsymbol 0\}\),
and hence
\[
{\rm RM}(F)=\mathbb P_0(F;\mathbb R).
\]
Here and below, the same expression for \(\boldsymbol x_F\) denotes its affine
extension to \(T\). The two components of \({\rm RM}(F)\) are
\(L^2(F)\)-orthogonal since \(\int_F\boldsymbol x_F\dd s=0\).
Moreover, for every \(\boldsymbol v\in {\rm RM}(T)\), its tangential
trace satisfies \(\Pi_F\boldsymbol v\in {\rm RM}(F)\).
Let \(Q_{{\rm RM},F}:L^2(F;\mathbb R^{d-1})\to {\rm RM}(F)\) denote the
\(L^2\)-orthogonal projector. When \(d=2\), \(Q_{{\rm RM},F}=Q_{0,F}\).

The space \(\mathbb P_0(T;\mathbb T)\) gives only constant tangential-normal
traces on each facet. For \(d\ge3\), the rotational component
\(\mathbb K_F\boldsymbol x_F\) of \({\rm RM}(F)\) is recovered by the
following linear enrichment modes from the geometric decomposition. Using the
indices \(j_\ell\) and
\(I_\ell\) fixed above, define
\[
\boldsymbol E_{\ell,\boldsymbol A}
:=
\sum_{i\in I_\ell}
\left(
\boldsymbol t_{i,j_\ell}^{\intercal}
\boldsymbol A\boldsymbol x_{F_\ell}
\right)
\dev(\nabla\lambda_i\otimes\boldsymbol t_{i,\ell}),
\qquad
\boldsymbol A\in\mathbb K_{F_\ell}.
\]
Define
\[
\mathbb P_0^+(T;\mathbb T)
:=
\mathbb P_0(T;\mathbb T)
+
\operatorname{span}
\left\{
\boldsymbol E_{\ell,\boldsymbol A}
:\ \boldsymbol A\in\mathbb K_{F_\ell},\quad \ell=0,\ldots,d
\right\}.
\]
For \(i\in I_\ell\), set
\(h_\ell:=-\boldsymbol t_{i,\ell}\cdot\boldsymbol n_\ell>0\), which is
independent of \(i\). Indeed,
\[
\tr_{F_m}^{tn}\dev(\nabla\lambda_i\otimes\boldsymbol t_{i,\ell})
=\Pi_{F_m}\nabla\lambda_i\,(\boldsymbol t_{i,\ell}\cdot\boldsymbol n_m),
\]
and, since \(\boldsymbol t_{i,j_\ell}\) is tangential to \(F_\ell\), the barycentric-coordinate identity
\[
\sum_{i\in I_\ell}
\Pi_{F_\ell}\nabla\lambda_i\,
\boldsymbol t_{i,j_\ell}^{\intercal}
=-\Pi_{F_\ell}
\]
gives
\[
\tr_{F_m}^{tn}\boldsymbol E_{\ell,\boldsymbol A}
=\delta_{m\ell}h_\ell\boldsymbol A\boldsymbol x_{F_\ell}.
\]
Consequently, the map
\(\boldsymbol A\mapsto\tr_{F_\ell}^{tn}\boldsymbol E_{\ell,\boldsymbol A}\)
is an isomorphism from \(\mathbb K_{F_\ell}\) onto
\(\mathbb K_{F_\ell}\boldsymbol x_{F_\ell}\). These rotational traces have
zero facet mean and therefore have trivial intersection with the constant
tangential-normal traces generated by \(\mathbb P_0(T;\mathbb T)\). In a
vanishing linear combination, facet averaging first gives that the constant
component has zero tangential-normal trace on every facet. The \(k=0\) case of
\eqref{eq:geometric} then yields that this component vanishes, and the trace
identity gives \(\boldsymbol A=0\) facet by facet. Hence the sum defining
\(\mathbb P_0^+(T;\mathbb T)\) is direct.

The same argument shows that the
tangential-normal trace map is injective on \(\mathbb P_0^+(T;\mathbb T)\), and
thus
\(\mathbb P_0^+(T;\mathbb T)\cap
\mathbb B_1^{tn}(T;\mathbb T)=\{0\}\).
We then define
\begin{equation}\label{eq:Sigma+TS0}
\Sigma_0^+(T;\mathbb S\cap\mathbb T)
:=
\left\{
\boldsymbol\tau\in
\mathbb P_0^+(T;\mathbb T)\oplus\mathbb B_1^{tn}(T;\mathbb T):
(\boldsymbol\tau,\boldsymbol q)_T=0
\quad
\forall\,\boldsymbol q\in\mathbb P_0(T;\mathbb K)
\right\}.
\end{equation}
The DoFs are
\begin{subequations}\label{eq:tndofTS0}
	\begin{align}
		(\boldsymbol t_{F,i}^{\intercal}\boldsymbol\tau\boldsymbol n,q)_F,
		&\qquad
		q\in\mathbb P_0(F),\;
		i=1,\ldots,d-1,\; F\in\mathcal F(T),
		\label{eq:tndofTS01}\\
		(\Pi_F(\boldsymbol\tau\boldsymbol n),\boldsymbol q)_F,
		&\qquad
		\boldsymbol q\in \mathbb K_F\boldsymbol x_F,\;
		F\in\mathcal F(T),
		\label{eq:tndofTS02}\\
		(\boldsymbol\tau,\boldsymbol q)_T,
		&\qquad
		\boldsymbol q\in
		\mathbb P_0(T;\mathbb S\cap\mathbb T).
		\label{eq:tndofTS03}
	\end{align}
\end{subequations}
The second family of facet DoFs is trivial when \(d=2\).

\begin{lemma}\label{le:unisolventTS0}
	The DoFs \eqref{eq:tndofTS0} are unisolvent for
	\(\Sigma_0^+(T;\mathbb S\cap\mathbb T)\).
\end{lemma}

\begin{proof}
	The weak-symmetry conditions are independent since
	\(\mathbb P_0(T;\mathbb K)\subset\mathbb P_0^+(T;\mathbb T)\)
	and the \(L^2(T)\)-pairing is nondegenerate.
	Together with the direct-sum
	properties above, this gives
	\(
	\dim\Sigma_0^+(T;\mathbb S\cap\mathbb T)
	=
	\frac{1}{2}d^2(d+1)-1.
	\)
	This agrees with the number of DoFs in \eqref{eq:tndofTS0}. It remains to prove uniqueness.
	
	Let \(\boldsymbol\tau\in\Sigma_0^+(T;\mathbb S\cap\mathbb T)\) have all
	DoFs in \eqref{eq:tndofTS0} equal to zero. By the trace property of the enrichment above, and since
	\(\mathbb B_1^{tn}(T;\mathbb T)\) has zero tangential-normal trace, the
	trace \(\tr_F^{tn}\boldsymbol\tau\) belongs to \({\rm RM}(F)\) on each facet.
	Since the decomposition of \({\rm RM}(F)\) is \(L^2(F)\)-orthogonal, the
	vanishing moments in \eqref{eq:tndofTS01}--\eqref{eq:tndofTS02} imply
	\(\tr_F^{tn}\boldsymbol\tau=0\) for $F\in\mathcal F(T)$.
	Thus \(\boldsymbol\tau\in\mathbb B_1^{tn}(T;\mathbb T)\). Moreover, the
	cell moments \eqref{eq:tndofTS03} and the weak-symmetry condition give
	\[
	(\boldsymbol\tau,\boldsymbol q)_T=0
	\quad
	\forall\,\boldsymbol q\in\mathbb P_0(T;\mathbb T),
	\]
	since \(\mathbb T=(\mathbb S\cap\mathbb T)\oplus\mathbb K\). Applying
	Lemma~\ref{le:unisolventT} with \(k=0\) yields
	\(\boldsymbol\tau=0\).
\end{proof}

For \(k\ge0\) and \(T\in\mathcal T_h\), we have
\begin{equation}\label{eq:Qkt2ST}
\mathbb P_k(T;\mathbb S\cap\mathbb T)
\subset \Sigma_k^+(T;\mathbb S\cap\mathbb T),\quad 
 Q_{k,T}\Sigma_k^+(T;\mathbb S\cap\mathbb T)=\mathbb P_k(T;\mathbb S\cap\mathbb T).	
\end{equation}
Indeed, the first inclusion follows from the \(L^2\)-orthogonality between symmetric and skew-symmetric tensors. For the second identity, the weak symmetry condition implies that \(Q_{k,T}\boldsymbol\tau\) has no component in \(\mathbb P_k(T;\mathbb K)\).

For \(k\ge0\), the local norm equivalence is
\begin{equation}\label{eq:local_stress_norm}
\|\boldsymbol\tau_h\|_{0,T}^2
\eqsim
\|Q_{k,T}\boldsymbol\tau_h\|_{0,T}^2
+\sum_{F\in\mathcal F(T)}
h_F\|\Pi_F\boldsymbol\tau_h\boldsymbol n_{\partial T}\|_{0,F}^2,\quad\forall\,\boldsymbol\tau_h\in
\Sigma_k^+(T;\mathbb S\cap\mathbb T).
\end{equation}
The constants depend only on \(d\), \(k\), and the mesh shape regularity.
By unisolvence, the right-hand side defines a norm on the local
finite-dimensional space. Uniform equivalence on a shape-regular family
follows by scaling and a standard compactness argument.

For \(d\ge3\), the additional facet moments \eqref{eq:tndofTS02} control
the rotational component in \({\rm RM}(F)\), which is required in the
discrete inf--sup argument based on Lemma~\ref{lem:facet_coupling_korn}.

\subsection{Global tensor space}
For each facet \(F\in\mathcal F_h\), set
\[
\mathcal R_k^t(F):=
\begin{cases}
\mathbb P_k(F;\mathbb R^{d-1}), & k\ge1,\\
{\rm RM}(F), & k=0,
\end{cases}
\]
and let \(Q_F^t:L^2(F;\mathbb R^{d-1})\to\mathcal R_k^t(F)\) be the
\(L^2(F)\)-orthogonal projector. Since
\(\boldsymbol n_F=\pm\boldsymbol n_{\partial T}\), the local constructions give, for every
\(\boldsymbol\tau\in\Sigma_k^+(T;\mathbb S\cap\mathbb T)\),
$\Pi_F(\boldsymbol\tau\boldsymbol n_F)\in\mathcal R_k^t(F)$ for
$F\in\mathcal F(T)$.
Using the fixed facet orientations
\(\boldsymbol n_F\), define the global finite element space for the viscous stress
\[
\begin{aligned}
\Sigma_h^{tn}:=\{\boldsymbol\tau_h\in \Sigma_h^{-1}:\;&[\![\Pi_F\boldsymbol\tau_h\boldsymbol n_F]\!]=\boldsymbol0
\quad\forall\,F\in\mathring{\mathcal F}_h
\},
\end{aligned}
\]
where $\Sigma_h^{-1}:=\prod_{T\in\mathcal T_h}
\Sigma_k^+(T;\mathbb S\cap\mathbb T)$.
Equivalently, all tangential-normal facet DoFs are single-valued across
interior facets. No boundary condition is imposed on these stress moments for
the no-slip problem.

\begin{remark}
The space \(\Sigma_h^{tn}\) is pointwise traceless but only weakly symmetric.
In contrast to strongly symmetric conforming tensor elements, its global
coupling consists solely of tangential-normal facet moments, while symmetry is
enforced locally through cell moments. In particular, the construction
requires no vertex degrees of freedom or interelement vertex constraints. In
general, \(\Sigma_h^{tn}\not\subset H(\div,\Omega;\mathbb T)\), since the full
normal trace need not be continuous.
\end{remark}

\section{Continuous Distributional Formulation}\label{sec:Brinkman}

We formulate the first-order Brinkman system
\eqref{eq:Brinkman1st_intro} as a continuous distributional mixed problem.
The analysis is based on a stable decomposition of
\((H_0(\div,\Omega))'\), which yields \(\nu\)-uniform well-posedness;
equivalence with the primal formulation is then established.

We introduce the space
\[
H^{-1}(\curl\div,\Omega;\mathbb S\cap\mathbb T)
:=
\bigl\{
\boldsymbol\tau\in L^2(\Omega;\mathbb S\cap\mathbb T):
\div\boldsymbol\tau\in(H_0(\div,\Omega))'
\bigr\},
\]
equipped with the norm
\(
\|\boldsymbol{\tau}\|_{H^{-1}(\curl\div)}^2
:=
\|\boldsymbol{\tau}\|_0^2
+
\|\div\boldsymbol{\tau}\|_{(H_0(\div,\Omega))'}^2,
\)
where
\[
\|\div\boldsymbol{\tau}\|_{(H_0(\div,\Omega))'}
:=
\sup_{\boldsymbol v\in H_0(\div,\Omega),\ \boldsymbol v\neq 0}
\frac{\langle\div\boldsymbol{\tau},\boldsymbol v\rangle}
{\|\boldsymbol v\|_{H(\div)}} .
\]
For \(d=3\), the condition \(\div\boldsymbol\tau\in (H_0(\div,\Omega))'\) is an equivalent dual characterization of the \(H^{-1}(\curl\div)\) regularity used in \cite[Section 2.2]{ChenHuangZhang2023}; we use this dual characterization to define the symmetric--traceless \(H^{-1}(\curl\div)\) stress space for all \(d\ge2\). For
\(q\in L^2(\Omega)\), define
\(\nabla q\in(H_0(\div,\Omega))'\) by
\(\langle\nabla q,\boldsymbol v\rangle:=-(q,\div\boldsymbol v)\) for
\(\boldsymbol v\in H_0(\div,\Omega)\).

Let \(\boldsymbol f\in(H_0(\div,\Omega))'\). A distributional mixed formulation of the first-order
system \eqref{eq:Brinkman1st_intro} seeks 
\((\boldsymbol\sigma,\boldsymbol u,p)\in 
H^{-1}(\curl\div,\Omega;\mathbb S\cap\mathbb T)
\times H_0(\div,\Omega)\times L_0^2(\Omega)\)
such that
\begin{subequations}\label{eq:BrinkmanMixed}
\begin{align}
a(\boldsymbol\sigma,p;\boldsymbol\tau,q)
+b(\boldsymbol\tau,q;\boldsymbol u)&=0,
\label{eq:BrinkmanMixed1}\\
b(\boldsymbol\sigma,p;\boldsymbol v)
-(\boldsymbol u,\boldsymbol v)
&=-\langle\boldsymbol f,\boldsymbol v\rangle,
\label{eq:BrinkmanMixed2}
\end{align}
for all
\(\boldsymbol\tau\in
H^{-1}(\curl\div,\Omega;\mathbb S\cap\mathbb T)\),
\(q\in L_0^2(\Omega)\), and
\(\boldsymbol v\in H_0(\div,\Omega)\).
\end{subequations}
Here
\[
a(\boldsymbol\sigma,p;\boldsymbol\tau,q)
:=\nu^{-1}(\boldsymbol\sigma,\boldsymbol\tau),
\qquad
b(\boldsymbol\tau,q;\boldsymbol v)
:=
\langle\div\boldsymbol\tau,\boldsymbol v\rangle
-(q,\div\boldsymbol v).
\]

\begin{lemma}\label{lem:continuous_infsup}
	The following stable decomposition holds:
	\begin{equation}\label{eq:decomposition}
		(H_0(\div,\Omega))'
		=
		\div H^{-1}(\curl\div,\Omega;\mathbb S\cap\mathbb T)
		+\nabla L_0^2(\Omega).
	\end{equation}
	That is, every \(\boldsymbol w\in (H_0(\div,\Omega))'\) can be written as
	\(\boldsymbol w=\div\boldsymbol\tau+\nabla q\), with
	\(\boldsymbol\tau\in
	H^{-1}(\curl\div,\Omega;\mathbb S\cap\mathbb T)\),
	\(q\in L_0^2(\Omega)\), and
	\[
	\|\boldsymbol\tau\|_{H^{-1}(\curl\div)}+\|q\|_0
	\lesssim
	\|\boldsymbol w\|_{(H_0(\div,\Omega))'} .
	\]
\end{lemma}

\begin{proof}
	Let \(\boldsymbol w\in (H_0(\div,\Omega))'\). Since
	\(H_0^1(\Omega;\mathbb R^d)\hookrightarrow H_0(\div,\Omega)\), the
	restriction of \(\boldsymbol w\) to \(H_0^1(\Omega;\mathbb R^d)\)
	belongs to \(H^{-1}(\Omega;\mathbb R^d)\), with
	\(
	\|\boldsymbol w\|_{H^{-1}(\Omega)}
	\lesssim
	\|\boldsymbol w\|_{(H_0(\div,\Omega))'} .
	\)
	By Korn's inequality and the Lax--Milgram theorem, there exists
	\(\boldsymbol z\in H_0^1(\Omega;\mathbb R^d)\) such that
	\[
	(\boldsymbol{\varepsilon}(\boldsymbol z),\boldsymbol{\varepsilon}(\boldsymbol v))
	=
	\langle\boldsymbol w,\boldsymbol v\rangle
	\qquad
	\forall\,\boldsymbol v\in H_0^1(\Omega;\mathbb R^d),
	\]
	and
	\[
	\|\boldsymbol{\varepsilon}(\boldsymbol z)\|_0
	\lesssim
	\|\boldsymbol w\|_{(H_0(\div,\Omega))'} .
	\]
	Set
	\(\boldsymbol\sigma:=-\boldsymbol{\varepsilon}(\boldsymbol z)
	\in L^2(\Omega;\mathbb S)\).
	Then, for every
	\(\boldsymbol v\in H_0^1(\Omega;\mathbb R^d)\),
	\[
	\langle\div\boldsymbol\sigma,\boldsymbol v\rangle
	=
	-(\boldsymbol\sigma,\grad\boldsymbol v)
	=
	(\boldsymbol{\varepsilon}(\boldsymbol z),\boldsymbol{\varepsilon}(\boldsymbol v))
	=
	\langle\boldsymbol w,\boldsymbol v\rangle .
	\]
	Hence
	\[
	\div\boldsymbol\sigma=\boldsymbol w
	\quad\text{in }H^{-1}(\Omega;\mathbb R^d),
	\qquad
	\|\boldsymbol\sigma\|_0
	\lesssim
	\|\boldsymbol w\|_{(H_0(\div,\Omega))'} .
	\]

	Set \(\boldsymbol\tau:=\dev\boldsymbol\sigma
	\in L^2(\Omega;\mathbb S\cap\mathbb T)\) and
	\(q:=d^{-1}\tr\boldsymbol\sigma\in L_0^2(\Omega)\), where the zero mean follows from \(\boldsymbol z\in H_0^1(\Omega;\mathbb R^d)\).
	Since
	\(\boldsymbol\sigma=\boldsymbol\tau+q\boldsymbol I\), we have
	\[
	\boldsymbol w=\div\boldsymbol\tau+\nabla q
	\quad\text{in }H^{-1}(\Omega;\mathbb R^d),
	\qquad
	\|\boldsymbol\tau\|_0+\|q\|_0
	\lesssim
	\|\boldsymbol w\|_{(H_0(\div,\Omega))'} .
	\]
	Moreover,
	\(\nabla q\in(H_0(\div,\Omega))'\) and
	\(\|\nabla q\|_{(H_0(\div,\Omega))'}\le \|q\|_0\).
	Therefore
	\(\div\boldsymbol\tau =
	\boldsymbol w-\nabla q
	\in(H_0(\div,\Omega))'\),
	so that
	\(\boldsymbol\tau\in
	H^{-1}(\curl\div,\Omega;\mathbb S\cap\mathbb T)\). Furthermore,
	\[
	\|\div\boldsymbol\tau\|_{(H_0(\div,\Omega))'}
	\le
	\|\boldsymbol w\|_{(H_0(\div,\Omega))'}
	+\|q\|_0.
	\]
	Combining the preceding estimates proves the asserted stable decomposition.
\end{proof}

Lemma~\ref{lem:continuous_infsup} immediately yields the continuous inf--sup
condition
\[
\|\boldsymbol v\|_{H(\div)}
\lesssim
\sup_{\substack{
	(\boldsymbol\tau,q)\in
	H^{-1}(\curl\div,\Omega;\mathbb S\cap\mathbb T)\times L_0^2(\Omega)\\
	(\boldsymbol\tau,q)\neq(\boldsymbol 0,0)}}
\frac{
	b(\boldsymbol\tau,q;\boldsymbol v)
}{
	\|\boldsymbol\tau\|_{H^{-1}(\curl\div)}+\|q\|_0
}
\;\;\;
\forall\,\boldsymbol v\in H_0(\div,\Omega).
\]
To obtain stability uniformly with respect to \(\nu\), we retain the standard
norms on \(H_0(\div,\Omega)\) and \(L_0^2(\Omega)\), and define
\begin{equation}\label{eq:BrinkmanContinuousStressNorm}
\|\boldsymbol\tau\|_{H^{-1}(\curl\div),\nu}^2
:=
\nu^{-1}\|\boldsymbol\tau\|_0^2
+\|\div\boldsymbol\tau\|_{(H_0(\div,\Omega))'}^2.
\end{equation}

\begin{lemma}\label{lem:BrinkmanContinuousCouplingEstimate}
	For every
	\((\boldsymbol\tau,q)\in
	H^{-1}(\curl\div,\Omega;\mathbb S\cap\mathbb T)\times L_0^2(\Omega)\),
	\begin{equation}\label{eq:BrinkmanContinuousCouplingEstimate}
	\|\div\boldsymbol\tau\|_{(H_0(\div,\Omega))'}+\|q\|_0
	\lesssim
	\|\boldsymbol\tau\|_0
	+
	\sup_{\boldsymbol v\in H_0(\div,\Omega),\ \boldsymbol v\ne0}
	\frac{b(\boldsymbol\tau,q;\boldsymbol v)}
	{\|\boldsymbol v\|_{H(\div)}}.
	\end{equation}
\end{lemma}

\begin{proof}
	The stable surjection
	\(\div H_0^1(\Omega;\mathbb R^d)=L_0^2(\Omega)\)
	\cite[Theorem~4.1]{AcostaDuranMuschietti2006} yields
	\(\boldsymbol v\in H_0^1(\Omega;\mathbb R^d)\) such that
	\(\div\boldsymbol v=q\) and
	\(\|\boldsymbol v\|_1\lesssim\|q\|_0\).
	Testing \(b(\boldsymbol\tau,q;\cdot)\) with \(\boldsymbol v\) and using
	\(\langle\div\boldsymbol\tau,\boldsymbol v\rangle
	=-(\boldsymbol\tau,\nabla\boldsymbol v)\) gives
	\[
	\|q\|_0\eqsim\sup_{\boldsymbol v\in H_0^1(\Omega;\mathbb R^d),\ \boldsymbol v\ne0}\frac{(\div\boldsymbol v,q)}{\|\boldsymbol{v}\|_1}
	\lesssim
	\|\boldsymbol\tau\|_0
	+
	\sup_{\boldsymbol v\in H_0(\div,\Omega),\ \boldsymbol v\ne0}
	\frac{b(\boldsymbol\tau,q;\boldsymbol v)}
	{\|\boldsymbol v\|_{H(\div)}}.
	\]
	The estimate for \(\div\boldsymbol\tau\) then follows directly from the
	definition of \(b\).
\end{proof}

\begin{theorem}\label{thm:BrinkmanContinuous}
	Let \(\boldsymbol f\in(H_0(\div,\Omega))'\). Problem
	\eqref{eq:BrinkmanMixed} admits a unique solution
	\((\boldsymbol\sigma,\boldsymbol u,p)\in
	H^{-1}(\curl\div,\Omega;\mathbb S\cap\mathbb T)
	\times H_0(\div,\Omega)\times L_0^2(\Omega)\).
	Moreover,
	\begin{equation}\label{eq:BrinkmanContinuousStability}
		\|\boldsymbol\sigma\|_{H^{-1}(\curl\div),\nu}+\|p\|_0
		+\|\boldsymbol u\|_{H(\div)}
		\lesssim
		\|\boldsymbol f\|_{(H_0(\div,\Omega))'}.
	\end{equation}
	If \(\boldsymbol f\in L^2(\Omega;\mathbb R^d)\), the right-hand side is
	bounded by \(\|\boldsymbol f\|_0\). 
\end{theorem}
\begin{proof}
	For any
	\((\boldsymbol\tau,q)\in
	H^{-1}(\curl\div,\Omega;\mathbb S\cap\mathbb T)\times L_0^2(\Omega)\),
	Lemma~\ref{lem:BrinkmanContinuousCouplingEstimate}, the boundedness of
	\(b\), and \(0<\nu\le1\) give
	\[
	a(\boldsymbol\tau,q;\boldsymbol\tau,q)
	+\sup_{\boldsymbol v\in H_0(\div,\Omega),\,\boldsymbol v\ne0}
	\frac{b(\boldsymbol\tau,q;\boldsymbol v)^2}
	{\|\boldsymbol v\|_{H(\div)}^2}
	\eqsim
	\|\boldsymbol\tau\|_{H^{-1}(\curl\div),\nu}^2+\|q\|_0^2.
	\]
	
	Moreover, for any \(\boldsymbol u\in H_0(\div,\Omega)\),
	\[
	\|\boldsymbol u\|_0^2
	+\sup_{\substack{(\boldsymbol\tau,q)\in
			H^{-1}(\curl\div,\Omega;\mathbb S\cap\mathbb T)\times L_0^2(\Omega)\\
			(\boldsymbol\tau,q)\ne0}}
	\frac{b(\boldsymbol\tau,q;\boldsymbol u)^2}
	{\|\boldsymbol\tau\|_{H^{-1}(\curl\div),\nu}^2+\|q\|_0^2}
	\eqsim
	\|\boldsymbol u\|_{H(\div)}^2.
	\]
	The upper bound follows from the boundedness of \(b\), while the lower
	bound follows by taking
	\((\boldsymbol\tau,q)=(\boldsymbol0,-\div\boldsymbol u)\), since
	\(\div\boldsymbol u\in L_0^2(\Omega)\).
	
	The preceding norm equivalences and
	\cite[Theorem~2.6]{Zulehner2011} yield unique solvability and
	\eqref{eq:BrinkmanContinuousStability}. If
	\(\boldsymbol f\in L^2(\Omega;\mathbb R^d)\), then
	\(\|\boldsymbol f\|_{(H_0(\div,\Omega))'}\le\|\boldsymbol f\|_0\).
\end{proof}

We next compare \eqref{eq:BrinkmanMixed} with the classical weak formulation,
which seeks
\((\boldsymbol u,p)\in H_0^1(\Omega;\mathbb R^d)\times L_0^2(\Omega)\) such that
\begin{subequations}\label{eq:BrinkmanPrimal}
\begin{align}
\nu(\boldsymbol{\varepsilon}(\boldsymbol u),\boldsymbol{\varepsilon}(\boldsymbol v))
+(\boldsymbol u,\boldsymbol v)
+(p,\div\boldsymbol v)
&=\langle\boldsymbol f,\boldsymbol v\rangle,
&&\forall\,\boldsymbol v\in H_0^1(\Omega;\mathbb R^d),
\label{eq:BrinkmanPrimal1}\\
(\div\boldsymbol u,q)&=0,
&&\forall\,q\in L_0^2(\Omega).
\label{eq:BrinkmanPrimal2}
\end{align}
\end{subequations}

\begin{theorem}\label{thm:BrinkmanEquivalence}
	Let
	\((\boldsymbol\sigma,\boldsymbol u,p)\in
	H^{-1}(\curl\div,\Omega;\mathbb S\cap\mathbb T)
	\times H_0(\div,\Omega)\times L_0^2(\Omega)\)
	solve \eqref{eq:BrinkmanMixed}. Then
	\(\boldsymbol\sigma=\nu\boldsymbol{\varepsilon}(\boldsymbol u)\), and
	\((\boldsymbol u,p)\in
	H_0^1(\Omega;\mathbb R^d)\times L_0^2(\Omega)\)
	solves \eqref{eq:BrinkmanPrimal}. Conversely, if
	\((\boldsymbol u,p)\in
	H_0^1(\Omega;\mathbb R^d)\times L_0^2(\Omega)\)
	solves \eqref{eq:BrinkmanPrimal} and
	\(\boldsymbol\sigma:=\nu\boldsymbol{\varepsilon}(\boldsymbol u)\), then
	\((\boldsymbol\sigma,\boldsymbol u,p)\) solves
	\eqref{eq:BrinkmanMixed}.
\end{theorem}

\begin{proof}
	Taking \(\boldsymbol\tau=\boldsymbol0\) and \(q=0\), respectively, in
	\eqref{eq:BrinkmanMixed1} gives
	\(\div\boldsymbol u=0\) and
	\(\nu^{-1}\boldsymbol\sigma
	=\dev\boldsymbol{\varepsilon}(\boldsymbol u)=\boldsymbol{\varepsilon}(\boldsymbol u)\)
	in distributions. Hence
	\(\boldsymbol u\in H^1(\Omega;\mathbb R^d)\) by
	\cite[Theorem~2.3 and Proposition~2.10]{AmroucheGirault1994}.

A facetwise \(H^1\)-lifting of arbitrary tangential boundary data into symmetric traceless tensors, applied to \eqref{eq:BrinkmanMixed1}, shows that the tangential trace of \(\boldsymbol u\) vanishes on \(\partial\Omega\). Since
	\(\boldsymbol u\in H_0(\div,\Omega)\), its normal trace also vanishes, and
	therefore
	\(\boldsymbol u\in H_0^1(\Omega;\mathbb R^d)\).
	Equation~\eqref{eq:BrinkmanMixed2} gives
	\eqref{eq:BrinkmanPrimal1}, while
	\eqref{eq:BrinkmanPrimal2} follows from
	\(\div\boldsymbol u=0\).

	Conversely, let \((\boldsymbol u,p)\) solve
	\eqref{eq:BrinkmanPrimal} and set
	\(\boldsymbol\sigma=\nu\boldsymbol{\varepsilon}(\boldsymbol u)\).
	Then \(\div\boldsymbol u=0\) and
	\[
	\div\boldsymbol\sigma+\nabla p
	=\boldsymbol u-\boldsymbol f
	\quad\text{in }H^{-1}(\Omega;\mathbb R^d).
	\]
	Since \(\boldsymbol u-\boldsymbol f\) and
	\(\nabla p\) belong to \((H_0(\div,\Omega))'\), it follows that
	\(\div\boldsymbol\sigma\in(H_0(\div,\Omega))'\), and hence
	\(\boldsymbol\sigma\in
	H^{-1}(\curl\div,\Omega;\mathbb S\cap\mathbb T)\).
	Integration by parts now yields \eqref{eq:BrinkmanMixed}.
\end{proof}

%
%
%
%
%

\section{The Distributional Mixed Element Method}\label{sec:BrinkmanDiscrete}

Throughout the discrete formulation and error analysis, we assume
\(\boldsymbol f\in L^2(\Omega;\mathbb R^d)\).
Using the tensor elements constructed in Section~\ref{sec:weaklysymmtrictracelesstnelements}, we discretize the
continuous distributional formulation of Section~\ref{sec:Brinkman} and establish the stability properties of the resulting mixed method.

\subsection{Discrete spaces and interpolation operators}

\subsubsection{Velocity--pressure spaces}
For \(k\geq0\), we use the \(H(\div)\)-conforming
Brezzi--Douglas--Marini (BDM) element of order \(k+1\) for the velocity
\cite{BrezziDouglasMarini1985,Nedelec1986,BrezziDouglasDuranFortin1987}.
Its local space on \(T\in\mathcal T_h\) is
\(\mathbb P_{k+1}(T;\mathbb R^d)\), with DoFs
\cite[Theorem~3.14]{ChenHuang2022}
\begin{subequations}\label{eq:HdivDoF}
	\begin{align}
		(\boldsymbol v\cdot\boldsymbol n,q)_F,
		&\qquad
		q\in\mathbb P_{k+1}(F),\;
		F\in\mathcal F(T),
		\label{Hdivfemdof1}\\
		(\boldsymbol v,\boldsymbol q)_T,
		&\qquad
		\boldsymbol q\in
		\grad\mathbb P_k(T)\oplus
		\mathbb P_{k-1}(T;\mathbb K)\boldsymbol x.
		\label{Hdivfemdof2}
	\end{align}
\end{subequations}

The corresponding global velocity spaces are
\begin{align*}
	\mathbb V_h^{\div}
	&:=
	\{\boldsymbol v_h\in H(\div,\Omega):
	\boldsymbol v_h|_T\in\mathbb P_{k+1}(T;\mathbb R^d)
	\quad\forall\,T\in\mathcal T_h\},\\
	\mathring{\mathbb V}_h^{\div}
	&:=
	\mathbb V_h^{\div}\cap H_0(\div,\Omega).
\end{align*}
The pressure is approximated by discontinuous piecewise polynomials:
\[
\mathcal Q_h
:=
\{q_h\in L^2(\Omega):
q_h|_T\in\mathbb P_k(T)
\quad\forall\,T\in\mathcal T_h\},
\qquad
\mathring{\mathcal Q}_h
:=
\mathcal Q_h\cap L_0^2(\Omega).
\]

\subsubsection{Interpolation and projections}

For each \(T\in\mathcal T_h\), let
\(I_{k+1,T}^{\rm div}:H^1(T;\mathbb R^d)\to \mathbb P_{k+1}(T;\mathbb R^d)\)
denote the local interpolation operator defined by the DoFs in
\eqref{eq:HdivDoF}. The corresponding global interpolation operator
\(I_{k+1,h}^{\div}:H^1(\mathcal T_h;\mathbb R^d)\to
L^2(\Omega;\mathbb R^d)\) is defined by
\[
(I_{k+1,h}^{\div}\boldsymbol v)|_T
:=
I_{k+1,T}^{\rm div}(\boldsymbol v|_T),
\qquad T\in\mathcal T_h.
\]
Then \(I_{k+1,h}^{\div}\boldsymbol v\in\mathbb V_h^{\div}\) for any
\(\boldsymbol v\in H^1(\Omega;\mathbb R^d)\). Moreover, for
\(\boldsymbol v\in H^m(T;\mathbb R^d)\), \(1\le m\le k+2\),
\begin{equation}\label{eq:Ihdivprop2}
	\|\boldsymbol v-I_{k+1,T}^{\rm div}\boldsymbol v\|_{0,T}
	+
	h_T|\boldsymbol v-I_{k+1,T}^{\rm div}\boldsymbol v|_{1,T}
	\lesssim
	h_T^m|\boldsymbol v|_{m,T}.
\end{equation}

The interpolation operator \(I_{k+1,h}^{\div}\) satisfies the commuting property
\cite[Proposition~2.5.2]{Boffi2013MixedFE}
\begin{equation}\label{eq:divcommutative}
	\div(I_{k+1,h}^{\div}\boldsymbol v)
	=
	Q_{k,h}(\div\boldsymbol v)
	\qquad
	\forall\,\boldsymbol v\in H^1(\Omega;\mathbb R^d).
\end{equation}
For \(\boldsymbol v\in H^1(\mathcal T_h;\mathbb R^d)\), define
\[
|\boldsymbol v|_{1,h}^2
:=
\|\dev\boldsymbol{\varepsilon}_h(\boldsymbol v)\|_0^2
+\|\div_h\boldsymbol v\|_0^2
+\sum_{F\in\mathcal F_h}
h_F^{-1}\|[\![\Pi_F\boldsymbol v]\!]\|_{0,F}^2.
\]
Together with the stable surjection
\(\div H_0^1(\Omega;\mathbb R^d)=L_0^2(\Omega)\),
\eqref{eq:divcommutative} gives
\[
\div\mathring{\mathbb V}_h^{\div}=\mathring{\mathcal Q}_h.
\]
Taking the BDM interpolant of a stable continuous divergence lifting and
using the standard local stability and trace estimates, we obtain, for
every \(q_h\in\mathring{\mathcal Q}_h\), a function
\(\boldsymbol w_h\in\mathring{\mathbb V}_h^{\div}\) such that
\begin{equation}\label{eq:BDMdivsurjection}
	\div\boldsymbol w_h=q_h,
	\qquad
	|\boldsymbol w_h|_{1,h}\lesssim\|q_h\|_0.
\end{equation}

Let
\(I_T^{tn}:H^1(T;\mathbb S\cap\mathbb T)\to
\Sigma_k^+(T;\mathbb S\cap\mathbb T)\)
denote the interpolation operator associated with the DoFs
\eqref{eq:tndofTS} for \(k\ge1\), and with the DoFs
\eqref{eq:tndofTS0} for \(k=0\). The corresponding global interpolation operator
\(I_h^{tn}:H^1(\mathcal T_h;\mathbb S\cap\mathbb T)\to
L^2(\Omega;\mathbb T)\) is defined by
$
(I_h^{tn}\boldsymbol\tau)|_T
:=
I_T^{tn}(\boldsymbol\tau|_T)$ for $T\in\mathcal T_h$.
For \(\boldsymbol\tau\in H^1(\Omega;\mathbb S\cap\mathbb T)\), the
tangential--normal facet moments are single-valued, and hence
\(I_h^{tn}\boldsymbol\tau\in\Sigma_h^{tn}\).

For \(\boldsymbol\tau\in L^2(T;\mathbb T)\), let
\(Q_T^\Sigma\boldsymbol\tau\in
\Sigma_k^+(T;\mathbb S\cap\mathbb T)\) be the \(L^2(T)\)-orthogonal
projection determined by
\[
(Q_T^\Sigma\boldsymbol\tau,\boldsymbol\eta_h)_T
=
(\boldsymbol\tau,\boldsymbol\eta_h)_T
\qquad
\forall\,\boldsymbol\eta_h\in
\Sigma_k^+(T;\mathbb S\cap\mathbb T).
\]
The corresponding elementwise projection is defined by
\((Q_h^\Sigma\boldsymbol\tau)|_T
:=Q_T^\Sigma(\boldsymbol\tau|_T)\) for
\(\boldsymbol\tau\in L^2(\Omega;\mathbb T)\).

\begin{lemma}\label{le:Ih-error}
	For \(1\le s\le k+1\) and
	\(\boldsymbol\tau\in H^s(T;\mathbb S\cap\mathbb T)\),
	\begin{equation}\label{eq:Itn-error}
		\|\boldsymbol\tau-I_T^{tn}\boldsymbol\tau\|_{0,T}
		+h_T|\boldsymbol\tau-I_T^{tn}\boldsymbol\tau|_{1,T}
		\lesssim
		h_T^s|\boldsymbol\tau|_{s,T}.
	\end{equation}
\end{lemma}

\begin{proof}
It follows from the norm equivalence \eqref{eq:local_stress_norm},
polynomial approximation, and standard scaling and inverse estimates.
\end{proof}

\subsection{Distributional mixed finite element method and stability}

To motivate the discrete bilinear form, let
\(\boldsymbol\tau\in H^1(\mathcal T_h;\mathbb M)\) satisfy
\([\![\Pi_F\boldsymbol\tau\boldsymbol n_F]\!]=\boldsymbol0\) on every
\(F\in\mathring{\mathcal F}_h\). Elementwise integration by parts and
tangential--normal continuity give
\begin{equation*}
\begin{aligned}
\langle\div\boldsymbol\tau,\boldsymbol v\rangle
&=\sum_{T\in\mathcal T_h}(\div\boldsymbol\tau,\boldsymbol v)_T
-\sum_{T\in\mathcal T_h}
(\boldsymbol n^{\intercal}\boldsymbol\tau\boldsymbol n,
\boldsymbol v\cdot\boldsymbol n)_{\partial T}
\end{aligned}
\qquad
\forall\,\boldsymbol v\in C_0^\infty(\Omega;\mathbb R^d).
\end{equation*}
Motivated by this identity and the continuous formulation
\eqref{eq:BrinkmanMixed}, we define the distributional mixed finite element
method as follows: find
\(
(\boldsymbol\sigma_h,\boldsymbol u_h,p_h)
\in\Sigma_h^{tn}\times\mathring{\mathbb V}_h^{\div}
\times\mathring{\mathcal Q}_h
\)
such that
\begin{subequations}\label{eq:BrinkmanDiscrete}
\begin{align}
\nu^{-1}(\boldsymbol\sigma_h,\boldsymbol\tau_h)
+b_h(\boldsymbol\tau_h,q_h;\boldsymbol u_h)&=0,
&&\forall\,\boldsymbol\tau_h\in\Sigma_h^{tn},\;
q_h\in\mathring{\mathcal Q}_h,
\label{eq:BrinkmanDiscrete1}\\
b_h(\boldsymbol\sigma_h,p_h;\boldsymbol v_h)
-(\boldsymbol u_h,\boldsymbol v_h)
&=-(\boldsymbol f,\boldsymbol v_h),
&&\forall\,\boldsymbol v_h\in\mathring{\mathbb V}_h^{\div}.
\label{eq:BrinkmanDiscrete2}
\end{align}
\end{subequations}
Here the bilinear form \(b_h(\cdot,\cdot;\cdot)\) is defined by
\[
\begin{aligned}
	b_h(\boldsymbol\tau_h,q_h;\boldsymbol v_h)
	:&=
	\sum_{T\in\mathcal T_h}
	\big[
	(\div\boldsymbol\tau_h,\boldsymbol v_h)_T
	-
	(\boldsymbol n^{\intercal}
	\boldsymbol\tau_h\boldsymbol n,
	\boldsymbol v_h\cdot\boldsymbol n)_{\partial T}
	\big]
	-(\div\boldsymbol v_h,q_h) \\
	&=
	\sum_{T\in\mathcal T_h}
	\big[
	-(\boldsymbol\tau_h,\dev\grad\boldsymbol v_h)_T
	+
	(\Pi_F\boldsymbol\tau_h\boldsymbol n,
	\Pi_F\boldsymbol v_h)_{\partial T}
	\big] \\
	&\quad
	-(\div\boldsymbol v_h,q_h).
\end{aligned}
\]
The second identity uses pointwise tracelessness.
Equivalently, the fixed facet orientations and the single-valued
tangential--normal stress moments give
\[
\begin{aligned}
	b_h(\boldsymbol\tau_h,q_h;\boldsymbol v_h)
	={}&
	-
	(\boldsymbol\tau_h,\dev\grad_h\boldsymbol v_h)
	+\sum_{F\in\mathcal F_h}
	(\Pi_F\boldsymbol\tau_h\boldsymbol n_F,
	[\![\Pi_F\boldsymbol v_h]\!])_F -(\div\boldsymbol v_h,q_h).
\end{aligned}
\]

For \(\boldsymbol\tau\in H^1(\mathcal T_h;\mathbb T)\), define
\[
\|\boldsymbol\tau\|_{0,h}^2
:=
\|\boldsymbol\tau\|_0^2
+\sum_{T\in\mathcal T_h}\sum_{F\in\mathcal F(T)}
h_F\bigl\|\Pi_F\bigl((\boldsymbol\tau|_T)
\boldsymbol n_{\partial T}\bigr)\bigr\|_{0,F}^2.
\]

Summing \eqref{eq:local_stress_norm} and using the inverse trace inequality
gives
\[
\|\boldsymbol\tau_h\|_{0,h}\eqsim\|\boldsymbol\tau_h\|_0
\qquad
\forall\,\boldsymbol\tau_h\in\Sigma_h^{tn}.
\]
The weak-symmetry property and the Cauchy--Schwarz inequality give
\[
|b_h(\boldsymbol\tau_h,q_h;\boldsymbol v_h)|
\lesssim
\bigl(\|\boldsymbol\tau_h\|_{0,h}+\|q_h\|_0\bigr)
|\boldsymbol v_h|_{1,h}.
\]

\begin{lemma}\label{lem:facet_coupling_korn}
	Every \(\boldsymbol v\in H^1(\mathcal T_h;\mathbb R^d)\cap H_0(\div,\Omega)\) satisfies
	\begin{equation}\label{eq:brokenkorn}
	\begin{aligned}
		\|\grad_h\boldsymbol v\|_0^2
		+\sum_{F\in\mathcal F_h}h_F^{-1}
		\|[\![\boldsymbol v]\!]\|_{0,F}^2
		&\eqsim
		\|\dev\boldsymbol{\varepsilon}_h(\boldsymbol v)\|_0^2
		+\|\div\boldsymbol v\|_0^2 \\
		&\quad
		+\sum_{F\in\mathcal F_h}h_F^{-1}
		\|Q_{{\rm RM},F}[\![\Pi_F\boldsymbol v]\!]\|_{0,F}^2.
	\end{aligned}
	\end{equation}
\end{lemma}

\begin{proof}
	Since \(\boldsymbol v\in H_0(\div,\Omega)\), its jump has vanishing
	normal component on every facet, and hence
	\([\![\boldsymbol v]\!]=[\![\Pi_F\boldsymbol v]\!]\).
	The tangential traces of elementwise rigid motions belong to
	\({\rm RM}(F)\). Since the jumps are purely tangential, their moments
	against rigid-motion traces are determined by
	\(Q_{{\rm RM},F}[\![\Pi_F\boldsymbol v]\!]\).
	Hence the projected discrete Korn inequality
	(cf. \cite[Lemma~2.10]{ChenHuangWangZhang2026} and
	\cite[Lemma~3.2]{ChenHuHuang2017}) gives
	\[
	\|\grad_h\boldsymbol v\|_0^2
	+\sum_{F\in\mathcal F_h}h_F^{-1}
	\|[\![\boldsymbol v]\!]\|_{0,F}^2
	\eqsim
	\|\boldsymbol{\varepsilon}_h(\boldsymbol v)\|_0^2
	+\sum_{F\in\mathcal F_h}
	h_F^{-1}\|Q_{{\rm RM},F}[\![\Pi_F\boldsymbol v]\!]\|_{0,F}^2.
	\]
	Since $\|\boldsymbol{\varepsilon}_h(\boldsymbol v)\|_0^2
	=
	\|\dev\boldsymbol{\varepsilon}_h(\boldsymbol v)\|_0^2
	+\frac1d\|\div\boldsymbol v\|_0^2$,
	the equivalence \eqref{eq:brokenkorn} follows.
\end{proof}

The broken Poincar\'e inequality and \eqref{eq:brokenkorn} imply
\begin{equation}\label{eq:discrete_poincare}
	\|\boldsymbol v\|_0
	\lesssim
	|\boldsymbol v|_{1,h}
	\qquad
	\forall\,\boldsymbol v\in
	H^1(\mathcal T_h;\mathbb R^d)\cap H_0(\div,\Omega).
\end{equation}

\begin{theorem}\label{thm:discrete_infsup}
For every \(\boldsymbol v_h\in\mathring{\mathbb V}_h^{\div}\),
	\begin{equation}\label{eq:discreteinfsup}
		|\boldsymbol v_h|_{1,h}
		\lesssim
		\sup_{(\boldsymbol\tau_h,q_h)\in
	\Sigma_h^{tn}\times\mathring{\mathcal Q}_h, (\boldsymbol\tau_h,q_h)\ne(0,0)}
		\frac{b_h(\boldsymbol\tau_h, q_h; \boldsymbol v_h)}
		{\|\boldsymbol\tau_h\|_{0,h} + \|q_h\|_0}.
	\end{equation}
\end{theorem}

\begin{proof}
Let \(\boldsymbol v_h\in\mathring{\mathbb V}_h^{\div}\).
Applying \eqref{eq:BDMdivsurjection} to \(\div\boldsymbol v_h\) gives
\(\boldsymbol w_h\in\mathring{\mathbb V}_h^{\div}\) such that
	\[
		\div\boldsymbol w_h=\div\boldsymbol v_h,
		\qquad
		|\boldsymbol w_h|_{1,h}\lesssim \|\div\boldsymbol v_h\|_0.
	\]
	Then
	\(\boldsymbol z_h:=\boldsymbol v_h-\boldsymbol w_h
	\in\mathring{\mathbb V}_h^{\div}\cap\ker(\div)\).

	Choose \(\boldsymbol\tau_h\in\Sigma_h^{tn}\) by
	\[
	\begin{aligned}
		(\Pi_F\boldsymbol\tau_h\boldsymbol n_F,\boldsymbol q)_F
		&=
		h_F^{-1}([\![\Pi_F\boldsymbol z_h]\!],\boldsymbol q)_F,
		&&\forall\,\boldsymbol q\in \mathcal R_k^t(F),\quad F\in\mathcal F_h,\\
		(\boldsymbol\tau_h,\boldsymbol q)_T
		&=
		-(\dev\boldsymbol{\varepsilon}_h(\boldsymbol z_h),\boldsymbol q)_T,
		&&\forall\,\boldsymbol q\in\mathbb P_k(T;\mathbb S\cap\mathbb T),
		\quad T\in\mathcal T_h.
	\end{aligned}
	\]
Local unisolvence and the single-valued facet prescription ensure that
these conditions uniquely determine \(\boldsymbol\tau_h\in\Sigma_h^{tn}\).
By \eqref{eq:Qkt2ST}, the prescribed moments give, on each
\(T\in\mathcal T_h\),
\[
Q_{k,T}(\boldsymbol\tau_h|_T)
=-(\dev\boldsymbol{\varepsilon}_h(\boldsymbol z_h))|_T,
\qquad
Q_F^t(\Pi_F\boldsymbol\tau_h\boldsymbol n_F)
=h_F^{-1}Q_F^t[\![\Pi_F\boldsymbol z_h]\!].
\]
The local norm equivalence \eqref{eq:local_stress_norm}, the prescribed
moments, and \(\div\boldsymbol z_h=0\) yield
\[
\|\boldsymbol\tau_h\|_{0,h}^2
\eqsim
\|\dev\boldsymbol{\varepsilon}_h(\boldsymbol z_h)\|_0^2
+\sum_{F\in\mathcal F_h}h_F^{-1}
\|Q_F^t[\![\Pi_F\boldsymbol z_h]\!]\|_{0,F}^2
\eqsim
|\boldsymbol z_h|_{1,h}^2,
\]
where the last equivalence follows from \eqref{eq:brokenkorn},
\({\rm RM}(F)\subset\mathcal R_k^t(F)\), and the \(L^2\)-stability of
\(Q_F^t\). Moreover, since
\(\skw\grad_h\boldsymbol z_h|_T\in\mathbb P_k(T;\mathbb K)\),
the weak-symmetry constraint gives
	\[
	(\boldsymbol\tau_h,\dev\grad_h\boldsymbol z_h)
	=
	(\boldsymbol\tau_h,\dev\boldsymbol{\varepsilon}_h(\boldsymbol z_h)).
	\]
	Thus, by the definition of \(b_h\),
	\[
	b_h(\boldsymbol\tau_h,0;\boldsymbol z_h)
	= \|\dev\boldsymbol{\varepsilon}_h(\boldsymbol z_h)\|_0^2 + \sum_{F\in\mathcal F_h}h_F^{-1}
		\|Q_F^t[\![\Pi_F\boldsymbol z_h]\!]\|_{0,F}^2 \eqsim
	|\boldsymbol z_h|_{1,h}^2.
	\]
	Using \(\boldsymbol z_h=\boldsymbol v_h-\boldsymbol w_h\), the identity above,
	the Cauchy--Schwarz inequality, and the bound for \(\boldsymbol w_h\), we obtain
	\[
	|\boldsymbol v_h|_{1,h}
	\le
	|\boldsymbol z_h|_{1,h}+|\boldsymbol w_h|_{1,h}
	\lesssim
		\sup_{\boldsymbol\tau_h\in\Sigma_h^{tn}, \boldsymbol\tau_h\ne\boldsymbol0}
		\frac{b_h(\boldsymbol\tau_h,0;\boldsymbol v_h)}
	{\|\boldsymbol\tau_h\|_{0,h}}
	+
	\|\div\boldsymbol v_h\|_0.
	\]
	Since \(\div\boldsymbol v_h\in\mathring{\mathcal Q}_h\), taking
	\(q_h=-\div\boldsymbol v_h\) gives
	\[
	\|\div\boldsymbol v_h\|_0
	\le
		\sup_{q_h\in\mathring{\mathcal Q}_h, q_h\ne0}
	\frac{b_h(0,q_h;\boldsymbol v_h)}{\|q_h\|_0}.
	\]
	Combining the last two estimates proves \eqref{eq:discreteinfsup}.
\end{proof}

We equip \(\mathring{\mathbb V}_h^{\div}\) and \(\Sigma_h^{tn}\) with the
following \(\nu\)-dependent norms:
\begin{align*}
	\|\boldsymbol v_h\|_{\nu,h}^2
	&:=\nu\|\dev\boldsymbol{\varepsilon}_h(\boldsymbol v_h)\|_0^2
	+\nu\sum_{F\in\mathcal F_h}h_F^{-1}
	\|[\![\Pi_F\boldsymbol v_h]\!]\|_{0,F}^2
	+\|\div\boldsymbol v_h\|_0^2+\|\boldsymbol v_h\|_0^2,\\
	\|\boldsymbol\tau_h\|_{\nu^{-1},h}^2
	&:=\nu^{-1}\|\boldsymbol\tau_h\|_{0,h}^2.
\end{align*}

For \((\boldsymbol\sigma_h,p_h,\boldsymbol u_h)\) and
\((\boldsymbol\tau_h,q_h,\boldsymbol v_h)\) in
\(\Sigma_h^{tn}\times\mathring{\mathcal Q}_h
\times\mathring{\mathbb V}_h^{\div}\), define the bilinear form
\[
\begin{aligned}
	A_{\nu,h}(\boldsymbol\sigma_h,p_h,\boldsymbol u_h;
	\boldsymbol\tau_h,q_h,\boldsymbol v_h)
	&:=\nu^{-1}(\boldsymbol\sigma_h,\boldsymbol\tau_h)
	+b_h(\boldsymbol\tau_h,q_h;\boldsymbol u_h)\\
	&\quad
	+b_h(\boldsymbol\sigma_h,p_h;\boldsymbol v_h)
	-(\boldsymbol u_h,\boldsymbol v_h).
\end{aligned}
\]

\begin{theorem}\label{thm:BrinkmanDiscrete}
Discrete method \eqref{eq:BrinkmanDiscrete} is uniformly well posed with respect to
\(0<\nu\le1\) in the above \(\nu\)-dependent norms, and its solution
\((\boldsymbol\sigma_h,\boldsymbol u_h,p_h)\) satisfies
\begin{equation}\label{eq:BrinkmanExactDivergence}
\div\boldsymbol u_h=0.
\end{equation}
Moreover, every
\((\boldsymbol\sigma_h,p_h,\boldsymbol u_h)\in
\Sigma_h^{tn}\times\mathring{\mathcal Q}_h\times
\mathring{\mathbb V}_h^{\div}\) satisfies
\begin{equation}\label{eq:BrinkmanDiscreteStability}
\|\boldsymbol\sigma_h\|_{\nu^{-1},h}+\|p_h\|_0
+\|\boldsymbol u_h\|_{\nu,h}
\lesssim
\sup_{(\boldsymbol\tau_h,q_h,\boldsymbol v_h)\ne0}
\frac{
A_{\nu,h}(\boldsymbol\sigma_h,p_h,\boldsymbol u_h;
\boldsymbol\tau_h,q_h,\boldsymbol v_h)}
{\|\boldsymbol\tau_h\|_{\nu^{-1},h}+\|q_h\|_0
+\|\boldsymbol v_h\|_{\nu,h}},
\end{equation}
	where \((\boldsymbol\tau_h,q_h,\boldsymbol v_h)\in
	\Sigma_h^{tn}\times\mathring{\mathcal Q}_h
	\times\mathring{\mathbb V}_h^{\div}\). The hidden constant is independent
of \(h\) and \(\nu\), but may depend on \(d\), \(k\), the domain, and the mesh
shape regularity.
\end{theorem}

\begin{proof}
	Taking
	\((\boldsymbol\tau_h,q_h,\boldsymbol v_h)
	=(\boldsymbol\sigma_h,p_h,-\boldsymbol u_h)\) gives
	\[
	A_{\nu,h}(\boldsymbol\sigma_h,p_h,\boldsymbol u_h;
	\boldsymbol\sigma_h,p_h,-\boldsymbol u_h)
	=\nu^{-1}\|\boldsymbol\sigma_h\|_0^2+\|\boldsymbol u_h\|_0^2.
	\]
	Theorem~\ref{thm:discrete_infsup} and
	\eqref{eq:local_stress_norm} give
	\[
	\nu^{1/2}|\boldsymbol u_h|_{1,h}
	\lesssim
	\sup_{(\boldsymbol\tau_h,q_h)\ne(0,0)\in
	\Sigma_h^{tn}\times\mathring{\mathcal Q}_h}
	\frac{A_{\nu,h}(\boldsymbol\sigma_h,p_h,\boldsymbol u_h;
	\boldsymbol\tau_h,q_h,\boldsymbol0)}
	{\|\boldsymbol\tau_h\|_{\nu^{-1},h}+\|q_h\|_0}
	+\nu^{-1/2}\|\boldsymbol\sigma_h\|_0.
	\]
	Taking \(q_h=-\div\boldsymbol u_h\) gives
	\[
	\|\div\boldsymbol u_h\|_0
	\le
	\sup_{q_h\in\mathring{\mathcal Q}_h, q_h\ne0}
	\frac{A_{\nu,h}(\boldsymbol\sigma_h,p_h,\boldsymbol u_h;
	\boldsymbol0,q_h,\boldsymbol0)}{\|q_h\|_0}.
	\]
	The stable discrete divergence surjection and
	\eqref{eq:discrete_poincare} yield
	\[
	\|p_h\|_0
	\lesssim
	\sup_{\boldsymbol v_h\in\mathring{\mathbb V}_h^{\div}, \boldsymbol v_h\ne\boldsymbol0}
	\frac{A_{\nu,h}(\boldsymbol\sigma_h,p_h,\boldsymbol u_h;
	\boldsymbol0,0,\boldsymbol v_h)}
	{\|\boldsymbol v_h\|_{\nu,h}}
	+\nu^{-1/2}\|\boldsymbol\sigma_h\|_{0,h}
	+\|\boldsymbol u_h\|_0.
	\]
	The preceding estimates and Young's inequality prove
	\eqref{eq:BrinkmanDiscreteStability}. Since the trial and test spaces
	have the same finite dimension, well-posedness follows.
Taking \(\boldsymbol\tau_h=\boldsymbol0\) and
\(q_h=\div\boldsymbol u_h\in\mathring{\mathcal Q}_h\) in
\eqref{eq:BrinkmanDiscrete1} gives \eqref{eq:BrinkmanExactDivergence}.
\end{proof}

%

Gradient perturbations leave the discrete velocity and stress unchanged.
\begin{corollary}\label{cor:BrinkmanPressureRobust}
	Let \((\boldsymbol\sigma_h,\boldsymbol u_h,p_h)\) solve
	\eqref{eq:BrinkmanDiscrete} with right-hand side \(\boldsymbol f\). For
	\(\phi\in L_0^2(\Omega)\), let
	\((\boldsymbol\sigma_h^\phi,\boldsymbol u_h^\phi,p_h^\phi)\) solve the
	same discrete problem with the second equation replaced by
	\[
	b_h(\boldsymbol\sigma_h^\phi,p_h^\phi;\boldsymbol v_h)
	-
	(\boldsymbol u_h^\phi,\boldsymbol v_h)
	=
	-(\boldsymbol f,\boldsymbol v_h)
	+
	(\phi,\div\boldsymbol v_h)
	\quad
	\forall\,\boldsymbol v_h\in\mathring{\mathbb V}_h^{\div}.
	\]
	Then \(\boldsymbol\sigma_h^\phi=\boldsymbol\sigma_h\),
	\(\boldsymbol u_h^\phi=\boldsymbol u_h\), and
	\(p_h^\phi=p_h-Q_{k,h}\phi\).
\end{corollary}

\begin{proof}
	Let \(\widehat p_h:=p_h-Q_{k,h}\phi\in\mathring{\mathcal Q}_h\).
	Since \(\div\boldsymbol v_h\in\mathring{\mathcal Q}_h\),
	\((\phi,\div\boldsymbol v_h)
	=(Q_{k,h}\phi,\div\boldsymbol v_h)\).
	Thus \((\boldsymbol\sigma_h,\boldsymbol u_h,\widehat p_h)\) solves the
	perturbed discrete problem, and uniqueness in
	Theorem~\ref{thm:BrinkmanDiscrete} proves the result.
\end{proof}

\begin{remark}[Lowest-order coupling]
	For the Stokes variant obtained by omitting the drag term, the standard
	mixed stability argument based on \eqref{eq:discreteinfsup} gives a
	well-posed method. In three dimensions, its lowest-order version and the
	minimal-facet MCS method
	\cite{GopalakrishnanKoglerLedererSchoeberl2023} both have six globally
	coupled scalar DoFs per facet. The latter employs a consistent
	vorticity--divergence stabilization, whereas the \({\rm RM}(F)\)-enriched
	stress space supplies the rotational trace control needed to close the
	discrete inf--sup argument without additional stabilization.
\end{remark}

\begin{remark}[Slip boundary conditions]\label{rem:slip_discrete}
For slip boundary conditions, replace \(\Sigma_h^{tn}\) by
\[
\mathring{\Sigma}_h^{tn}
:=
\bigl\{
\boldsymbol\tau_h\in\Sigma_h^{tn}:
\Pi_F\boldsymbol\tau_h\boldsymbol n_F=\boldsymbol0
\quad\forall\,F\in\mathcal F_h^\partial
\bigr\},
\]
while keeping \(\mathring{\mathbb V}_h^{\div}\) and
\(\mathring{\mathcal Q}_h\) unchanged and omitting the boundary-facet terms
from the velocity seminorm. The tangential traction condition is then imposed
through the stress space. The stability argument carries over using the
corresponding projected Korn inequality modulo rigid motions; the velocity
mass term controls the resulting finite-dimensional kernel.
\end{remark}

\section{Error Analysis}\label{sec:error}

We establish optimal-order error estimates in the natural norms and, by
comparison with the Darcy limit, a parameter-uniform boundary-layer estimate.
We use \(b_h\) and \(A_{\nu,h}\) also for their natural elementwise extensions
whenever well defined, and let \((\boldsymbol\sigma,\boldsymbol u,p)\) and
\((\boldsymbol\sigma_h,\boldsymbol u_h,p_h)\) denote the solutions of
\eqref{eq:BrinkmanMixed} and \eqref{eq:BrinkmanDiscrete}, respectively.

\subsection{Optimal-order error estimates}

For \(\boldsymbol w\in H^1(\Omega;\mathbb R^d)\) and
\(\boldsymbol\tau\in L^2(\Omega;\mathbb T)\), define the consistency functional
\[
\mathcal E_h(\boldsymbol w;\boldsymbol\tau)
:=
-\bigl((I-Q_{k,h})\skw\grad\boldsymbol w,
\boldsymbol\tau\bigr).
\]

\begin{lemma}\label{lem:BrinkmanMixedConsistency}
	For every
	\((\boldsymbol\tau_h,q_h,\boldsymbol v_h)\in
	\Sigma_h^{tn}\times\mathring{\mathcal Q}_h
	\times\mathring{\mathbb V}_h^{\div}\),
	\begin{equation}\label{eq:BrinkmanMixedConsistency}
		A_{\nu,h}(\boldsymbol\sigma,p,\boldsymbol u;
		\boldsymbol\tau_h,q_h,\boldsymbol v_h)
		=
		-(\boldsymbol f,\boldsymbol v_h)
		+\mathcal E_h(\boldsymbol u;\boldsymbol\tau_h).
	\end{equation}
	If \(\boldsymbol u\in H^{k+2}(\Omega;\mathbb R^d)\), then
	\begin{equation}\label{eq:BrinkmanMixedConsistencyEstimate}
		|\mathcal E_h(\boldsymbol u;\boldsymbol\tau_h)|
		\lesssim
		h^{k+1}|\boldsymbol u|_{k+2}
		\|\boldsymbol\tau_h\|_0.
	\end{equation}
\end{lemma}

\begin{proof}
	By Theorem~\ref{thm:BrinkmanEquivalence},
	\(\boldsymbol\sigma=\nu\boldsymbol{\varepsilon}(\boldsymbol u)\) and
	\(\div\boldsymbol u=0\). Hence, by the definition of \(b_h\),
	\[
	\begin{aligned}
		\nu^{-1}(\boldsymbol\sigma,\boldsymbol\tau_h)
		+b_h(\boldsymbol\tau_h,q_h;\boldsymbol u)
		&=
		(\boldsymbol{\varepsilon}(\boldsymbol u)-\dev\grad\boldsymbol u,
		\boldsymbol\tau_h)\\
		&=
		-\bigl((I-Q_{k,h})\skw\grad\boldsymbol u,
		\boldsymbol\tau_h\bigr),
	\end{aligned}
	\]
	where the last identity follows from
	\(Q_{k,h}\skw\grad\boldsymbol u\in
	\mathbb P_k(\mathcal T_h;\mathbb K)\) and the weak symmetry of
	\(\boldsymbol\tau_h\). Together with
	\[
	b_h(\boldsymbol\sigma,p;\boldsymbol v_h)
	-(\boldsymbol u,\boldsymbol v_h)
	=-(\boldsymbol f,\boldsymbol v_h),
	\]
	this proves \eqref{eq:BrinkmanMixedConsistency}. Estimate
	\eqref{eq:BrinkmanMixedConsistencyEstimate} follows from the approximation
	property of \(Q_{k,h}\) and the Cauchy--Schwarz inequality.
\end{proof}

For brevity, set
\[
\boldsymbol e_h^\sigma
:=I_h^{tn}\boldsymbol\sigma-\boldsymbol\sigma_h,
\qquad
e_h^p:=Q_{k,h}p-p_h,
\qquad
\boldsymbol e_h^u
:=I_{k+1,h}^{\div}\boldsymbol u-\boldsymbol u_h.
\]

\begin{theorem}\label{thm:BrinkmanError}
	Assume \(\boldsymbol u\in H^{k+2}(\Omega;\mathbb R^d)\). Then
	\begin{equation}\label{eq:BrinkmanProjectedError}
		\|\boldsymbol e_h^\sigma\|_{\nu^{-1},h}
		+\|e_h^p\|_0
		+\|\boldsymbol e_h^u\|_{\nu,h}
		\lesssim
		h^{k+1}(\nu^{1/2}+h)|\boldsymbol u|_{k+2}.
	\end{equation}
\end{theorem}

\begin{proof}
	By the commuting property \eqref{eq:divcommutative}, exact mass
	conservation \eqref{eq:BrinkmanExactDivergence}, and the
	\(L^2\)-orthogonality of \(Q_{k,h}\),
	\[
	\div\boldsymbol e_h^u=0,
	\qquad
	(Q_{k,h}p-p,\div\boldsymbol v_h)=0
	\quad\forall\,\boldsymbol v_h\in\mathring{\mathbb V}_h^{\div}.
	\]
	Subtracting the discrete equations from
	\eqref{eq:BrinkmanMixedConsistency} therefore gives
	\[
	\begin{aligned}
		&A_{\nu,h}(\boldsymbol e_h^\sigma,e_h^p,\boldsymbol e_h^u;
		\boldsymbol\tau_h,q_h,\boldsymbol v_h)\\
		&\quad=
		\nu^{-1}(I_h^{tn}\boldsymbol\sigma-\boldsymbol\sigma,
		\boldsymbol\tau_h)
		+b_h(\boldsymbol\tau_h,0;
		I_{k+1,h}^{\div}\boldsymbol u-\boldsymbol u)\\
		&\qquad
		+b_h(I_h^{tn}\boldsymbol\sigma-\boldsymbol\sigma,0;
		\boldsymbol v_h)
		-(I_{k+1,h}^{\div}\boldsymbol u-\boldsymbol u,
		\boldsymbol v_h)
		+\mathcal E_h(\boldsymbol u;\boldsymbol\tau_h).
	\end{aligned}
	\]
	Using \(\boldsymbol\sigma=\nu\boldsymbol{\varepsilon}(\boldsymbol u)\), the interpolation
	estimates and a standard trace argument give
	\[
	\nu^{-1}\|I_h^{tn}\boldsymbol\sigma-\boldsymbol\sigma\|_{0,h}
	+|I_{k+1,h}^{\div}\boldsymbol u-\boldsymbol u|_{1,h}
	+h^{-1}\|I_{k+1,h}^{\div}\boldsymbol u-\boldsymbol u\|_0
	\lesssim h^{k+1}|\boldsymbol u|_{k+2}.
	\]
	Together with \eqref{eq:BrinkmanMixedConsistencyEstimate} and the
	continuity of \(b_h\), this yields
	\[
	|A_{\nu,h}(\boldsymbol e_h^\sigma,e_h^p,
	\boldsymbol e_h^u;
	\boldsymbol\tau_h,q_h,\boldsymbol v_h)|
	\lesssim
	h^{k+1}(\nu^{1/2}+h)|\boldsymbol u|_{k+2}
	\bigl(\|\boldsymbol\tau_h\|_{\nu^{-1},h}
	+\|\boldsymbol v_h\|_{\nu,h}\bigr).
	\]
	Applying \eqref{eq:BrinkmanDiscreteStability} proves
	\eqref{eq:BrinkmanProjectedError}.
\end{proof}

\begin{corollary}\label{cor:BrinkmanExactError}
 Under the assumptions of Theorem~\ref{thm:BrinkmanError}, suppose further that
	\(p\in H^{k+1}(\Omega)\). Then
	\begin{equation}\label{eq:BrinkmanExactError}
		\begin{aligned}
			\nu^{-1/2}\|\boldsymbol\sigma-\boldsymbol\sigma_h\|_{0,h}
			+\nu^{1/2}|\boldsymbol u-\boldsymbol u_h|_{1,h}
			+\|\boldsymbol u-\boldsymbol u_h\|_0
			\lesssim h^{k+1}(\nu^{1/2}+h)|\boldsymbol u|_{k+2},&\\
			\|p-p_h\|_0
			\lesssim h^{k+1}
			\bigl((\nu^{1/2}+h)|\boldsymbol u|_{k+2}
			+|p|_{k+1}\bigr).&
		\end{aligned}
	\end{equation}
\end{corollary}

\begin{proof}
	The result follows from Theorem~\ref{thm:BrinkmanError}, the interpolation
	estimates, and the triangle inequality.
\end{proof}

\subsection{Parameter-uniform boundary-layer estimate}

The preceding estimates involve higher Sobolev norms that may deteriorate
as \(\nu\to0\) owing to boundary layers. We therefore compare with the Darcy
limit to derive a parameter-uniform estimate.

Throughout this subsection, assume further that \(\Omega\) is convex and
\(\boldsymbol f\in H(\curl,\Omega)\). We adopt the arbitrary-dimensional
convention
\(\curl\boldsymbol v:=\grad\boldsymbol v-(\grad\boldsymbol v)^{\intercal}\)
from \cite[Section~2]{HuangWang2023} and define
\[
H(\curl,\Omega)
:=
\{\boldsymbol v\in L^2(\Omega;\mathbb R^d):
\curl\boldsymbol v\in L^2(\Omega;\mathbb K)\}.
\]
We equip this space with the norm
\[
\|\boldsymbol v\|_{H(\curl)}^2
:=
\|\boldsymbol v\|_0^2+\|\curl\boldsymbol v\|_0^2.
\]
Let
\((\boldsymbol u^0,p^0)\in
H_0(\div,\Omega)\times L_0^2(\Omega)\) solve the Darcy limit problem
\begin{equation}\label{eq:BrinkmanLimitProblem}
\boldsymbol u^0-\nabla p^0=\boldsymbol f, \;\;
\div\boldsymbol u^0=0\;\quad\text{in }\Omega; \quad
\boldsymbol u^0\cdot\boldsymbol n=0 \;\quad\text{on }\partial\Omega.
\end{equation}
We further assume the parameter-explicit regularity estimates
\begin{equation}\label{eq:BrinkmanBoundaryRegularity}
	\begin{aligned}
		\nu\|\boldsymbol u\|_2+\nu^{1/2}\|\boldsymbol u\|_1
		+\|\boldsymbol u-\boldsymbol u^0\|_0+\|p-p^0\|_1
		&\lesssim\nu^{1/4}\|\boldsymbol f\|_{H(\curl)},\\
		\|\boldsymbol u^0\|_1+\|p^0\|_1
		&\lesssim\|\boldsymbol f\|_{H(\curl)}.
	\end{aligned}
\end{equation}
For \(d=2\), estimates of the form
\eqref{eq:BrinkmanBoundaryRegularity} are proved in
\cite[(6.12) and (6.15)]{MardalTaiWinther2002}; related
three-dimensional estimates under stronger regularity assumptions are
given in \cite{TaiWinther2006}.

To exploit the \(L^2\)-control of the boundary-layer remainder
\(\boldsymbol u-\boldsymbol u^0\), we use a locally \(L^2\)-bounded
commuting projection onto a Raviart--Thomas subspace
\cite{RaviartThomas1977,Boffi2013MixedFE}. Set
\[
\mathring{\mathbb{RT}}_{k,h}
:=
\left\{
\boldsymbol v_h\in H_0(\div,\Omega):
\boldsymbol v_h|_T\in
\mathbb P_k(T;\mathbb R^d)+\boldsymbol x\mathbb P_k(T)
\quad\forall\,T\in\mathcal T_h
\right\}.
\]
Then
\(\mathring{\mathbb{RT}}_{k,h}\subset
\mathring{\mathbb V}_h^{\div}\).
Let \(P_h^{\rm RT}:H_0(\div,\Omega)\to
\mathring{\mathbb{RT}}_{k,h}\) be the projection of
\cite[Definition~3.1 and Theorem~3.2]{ErnGudiSmearsVohralik2022}
with \(\Gamma_N=\partial\Omega\) and polynomial degree \(k\).
The construction is valid in arbitrary space dimensions; see
\cite[Section~1.6]{ErnGudiSmearsVohralik2022}.

By \cite[Theorem~3.2 and (3.9a)]{ErnGudiSmearsVohralik2022},
\[
\div P_h^{\rm RT}\boldsymbol v=Q_{k,h}(\div\boldsymbol v),
\qquad
\|P_h^{\rm RT}\boldsymbol v\|_0\lesssim\|\boldsymbol v\|_0
\quad\text{if }\div\boldsymbol v=0.
\]

\begin{lemma}\label{lem:RTProjectionProperties}
	For
	\(\boldsymbol v\in H^{k+1}(\Omega;\mathbb R^d)
	\cap H_0(\div,\Omega)\cap\ker(\div)\),
	\[
	\|(I-P_h^{\rm RT})\boldsymbol v\|_0
	\lesssim h^{k+1}|\boldsymbol v|_{k+1}.
	\]
	Moreover,
	\begin{equation}\label{eq:PhRTStability}
	|P_h^{\rm RT}\boldsymbol v|_{1,h}
	\lesssim|\boldsymbol v|_1
	\qquad
	\forall\,\boldsymbol v\in
	H_0^1(\Omega;\mathbb R^d)\cap\ker(\div).
	\end{equation}
\end{lemma}

\begin{proof}
	Let \(I_h^{\rm RT}\) be the canonical Raviart--Thomas interpolant.
	Since
	\(\div(\boldsymbol v-I_h^{\rm RT}\boldsymbol v)=0\) and
	\(P_h^{\rm RT}I_h^{\rm RT}\boldsymbol v=I_h^{\rm RT}\boldsymbol v\),
	\[
	\|(I-P_h^{\rm RT})\boldsymbol v\|_0
	\lesssim
	\|\boldsymbol v-I_h^{\rm RT}\boldsymbol v\|_0
	\lesssim h^{k+1}|\boldsymbol v|_{k+1}.
	\]
	For the second estimate, the decomposition
	\(P_h^{\rm RT}\boldsymbol v
	=I_h^{\rm RT}\boldsymbol v
	+P_h^{\rm RT}(\boldsymbol v-I_h^{\rm RT}\boldsymbol v)\),
	together with local stability and scaling, gives \eqref{eq:PhRTStability}.
\end{proof}

\begin{lemma}\label{lem:BrinkmanDarcyComparison}
	Assume that
	\(\boldsymbol u^0\in H^{k+1}(\Omega;\mathbb R^d)\), and let
	\(\boldsymbol w_h:=P_h^{\rm RT}\boldsymbol u\). Then
	\begin{equation}\label{eq:BrinkmanComparisonApproximation}
		\|\boldsymbol u-\boldsymbol w_h\|_0
		+\|\boldsymbol u^0-\boldsymbol w_h\|_0
		\lesssim
		\|\boldsymbol u-\boldsymbol u^0\|_0
		+h^{k+1}|\boldsymbol u^0|_{k+1}.
	\end{equation}
	Moreover,
	\begin{equation}\label{eq:BrinkmanDarcyComparison}
		\begin{aligned}
			&\nu^{-1/2}\|\boldsymbol\sigma_h\|_0
			+\nu^{1/2}|\boldsymbol w_h-\boldsymbol u_h|_{1,h}
			+\|\boldsymbol w_h-\boldsymbol u_h\|_0
			+\|Q_{k,h}p^0-p_h\|_0\\
			&\qquad\lesssim
			\nu^{1/2}|\boldsymbol u|_1
			+\|\boldsymbol u-\boldsymbol u^0\|_0
			+h^{k+1}|\boldsymbol u^0|_{k+1}.
		\end{aligned}
	\end{equation}
\end{lemma}

\begin{proof}
	The commuting property and Lemma~\ref{lem:RTProjectionProperties} give
	\(\div\boldsymbol w_h=0\) and
	\(|\boldsymbol w_h|_{1,h}\lesssim|\boldsymbol u|_1\).
	Since \(\boldsymbol u-\boldsymbol u^0\) is divergence-free, the
	\(L^2\)-stability of \(P_h^{\rm RT}\) and
	Lemma~\ref{lem:RTProjectionProperties} give
	\[
	\|\boldsymbol u-\boldsymbol w_h\|_0
	\lesssim
	\|\boldsymbol u-\boldsymbol u^0\|_0
	+h^{k+1}|\boldsymbol u^0|_{k+1}.
	\]
	The triangle inequality proves
	\eqref{eq:BrinkmanComparisonApproximation}.
	
	The discrete equations and \eqref{eq:BrinkmanLimitProblem} give
	\[
	A_{\nu,h}(\boldsymbol\sigma_h,p_h-Q_{k,h}p^0,
	\boldsymbol u_h-\boldsymbol w_h;
	\boldsymbol\tau_h,q_h,\boldsymbol v_h)
	=
	-b_h(\boldsymbol\tau_h,0;\boldsymbol w_h)
	+(\boldsymbol w_h-\boldsymbol u^0,\boldsymbol v_h).
	\]
	The stability estimate \eqref{eq:BrinkmanDiscreteStability} yields
	\[
	\begin{aligned}
		&\nu^{-1/2}\|\boldsymbol\sigma_h\|_0
		+\nu^{1/2}|\boldsymbol w_h-\boldsymbol u_h|_{1,h}
		+\|\boldsymbol w_h-\boldsymbol u_h\|_0
		+\|Q_{k,h}p^0-p_h\|_0\\
		&\qquad\lesssim
		\nu^{1/2}|\boldsymbol w_h|_{1,h}
		+\|\boldsymbol w_h-\boldsymbol u^0\|_0.
	\end{aligned}
	\]
	Using \eqref{eq:BrinkmanComparisonApproximation} proves
	\eqref{eq:BrinkmanDarcyComparison}.
\end{proof}

\begin{theorem}\label{thm:BrinkmanBoundaryLayerOptimalError}
	Suppose that
	\(\boldsymbol u^0\in H^{k+1}(\Omega;\mathbb R^d)\) and
	\(p^0\in H^{k+1}(\Omega)\). Then
	\begin{equation}\label{eq:BrinkmanBoundaryLayerOptimalError}
		\begin{aligned}
			&\nu^{-1/2}\|\boldsymbol\sigma-\boldsymbol\sigma_h\|_0
			+\nu^{1/2}|\boldsymbol u-\boldsymbol u_h|_{1,h}
			+\|\boldsymbol u-\boldsymbol u_h\|_0+\|p-p_h\|_0\\
			&\qquad\lesssim
			\nu^{1/4}\|\boldsymbol f\|_{H(\curl)}
			+h^{k+1}
			\bigl(
			|\boldsymbol u^0|_{k+1}+|p^0|_{k+1}
			\bigr).
		\end{aligned}
	\end{equation}
	Moreover,
	\begin{equation}\label{eq:BrinkmanDarcyLimitError}
		\begin{aligned}
			&\nu^{1/2}
			\bigl\|Q_h^\Sigma(\dev\grad\boldsymbol u^0)
			-\nu^{-1}\boldsymbol\sigma_h\bigr\|_0
			+\|\boldsymbol u^0-\boldsymbol u_h\|_0
			+\|p^0-p_h\|_0\\
			&\qquad\lesssim
			\nu^{1/4}\|\boldsymbol f\|_{H(\curl)}
			+h^{k+1}
			\bigl(
			|\boldsymbol u^0|_{k+1}+|p^0|_{k+1}
			\bigr).
		\end{aligned}
	\end{equation}
\end{theorem}

\begin{proof}
	Let \(\boldsymbol w_h=P_h^{\rm RT}\boldsymbol u\). By
	\(|\boldsymbol w_h|_{1,h}\lesssim|\boldsymbol u|_1\), the triangle
	inequality gives
	\[
	\begin{aligned}
		&\nu^{-1/2}\|\boldsymbol\sigma-\boldsymbol\sigma_h\|_0
		+\nu^{1/2}|\boldsymbol u-\boldsymbol u_h|_{1,h}
		+\|\boldsymbol u-\boldsymbol u_h\|_0+\|p-p_h\|_0\\
		&\quad\lesssim
		\nu^{1/2}|\boldsymbol u|_1
		+\|\boldsymbol u-\boldsymbol w_h\|_0
		+\|p-p^0\|_0+h^{k+1}|p^0|_{k+1}\\
		&\qquad
		+\nu^{-1/2}\|\boldsymbol\sigma_h\|_0
		+\nu^{1/2}|\boldsymbol w_h-\boldsymbol u_h|_{1,h}
		+\|\boldsymbol w_h-\boldsymbol u_h\|_0
		+\|Q_{k,h}p^0-p_h\|_0.
	\end{aligned}
	\]
	Applying \eqref{eq:BrinkmanComparisonApproximation} and
	\eqref{eq:BrinkmanDarcyComparison} to this bound, followed by
	\eqref{eq:BrinkmanBoundaryRegularity}, yields
	\eqref{eq:BrinkmanBoundaryLayerOptimalError}.
	
	Since \(Q_h^\Sigma\boldsymbol\sigma_h=\boldsymbol\sigma_h\) and
	\(\boldsymbol\sigma=\nu\boldsymbol{\varepsilon}(\boldsymbol u)\), the
	\(L^2\)-stability of \(Q_h^\Sigma\) gives
	\[
	\begin{aligned}
		\nu^{1/2}
		\bigl\|Q_h^\Sigma(\dev\grad\boldsymbol u^0)
		-\nu^{-1}\boldsymbol\sigma_h\bigr\|_0\le
		\nu^{1/2}\|\dev\grad\boldsymbol u^0
		-\boldsymbol{\varepsilon}(\boldsymbol u)\|_0
		+\nu^{-1/2}\|\boldsymbol\sigma-\boldsymbol\sigma_h\|_0.
	\end{aligned}
	\]
	Estimate \eqref{eq:BrinkmanDarcyLimitError} follows from
	\eqref{eq:BrinkmanBoundaryRegularity} and
	\eqref{eq:BrinkmanBoundaryLayerOptimalError}.
\end{proof}

\section{Equivalent formulations}\label{sub:equivformulations}

The distributional mixed method \eqref{eq:BrinkmanDiscrete} admits two
algebraically equivalent formulations. We first derive a stress-hybridized
formulation and then present an equivalent stabilization-free virtual element
formulation.

\subsection{Stress-hybridized formulation}

We relax the tangential--normal stress continuity in
\eqref{eq:BrinkmanDiscrete}. 
The corresponding tangential facet multiplier space is
\[
\Lambda_h^t
:=
\bigl\{
\boldsymbol\lambda_h\in L^2(\mathcal F_h;\mathbb R^{d-1}):
\boldsymbol\lambda_h|_F\in \mathcal R_k^t(F)\quad \forall\,F\in\mathring{\mathcal F}_h,\;
\boldsymbol\lambda_h|_F=\boldsymbol0\quad \forall\,F\in\mathcal F_h^\partial
\bigr\}.
\]
For
\((\boldsymbol v_h,\boldsymbol\lambda_h)\in
H^1(\mathcal T_h;\mathbb R^d)\times\Lambda_h^t\), define
\(\dev\boldsymbol{\varepsilon}_w(\boldsymbol v_h,\boldsymbol\lambda_h)\in\Sigma_h^{-1}\) by
\[
\begin{aligned}
	(\dev\boldsymbol{\varepsilon}_w(\boldsymbol v_h,\boldsymbol\lambda_h),\boldsymbol\tau_h)
	:={}&
	\sum_{T\in\mathcal T_h}\Bigl[
	-(\boldsymbol v_h,\div\boldsymbol\tau_h)_T
	+(\boldsymbol v_h\cdot\boldsymbol n_{\partial T},
	\boldsymbol n_{\partial T}^{\intercal}\boldsymbol\tau_h
	\boldsymbol n_{\partial T})_{\partial T}\\
	&\qquad+
	\sum_{F\in\mathcal F(T)}
	(\boldsymbol\lambda_h,\Pi_F\boldsymbol\tau_h\boldsymbol n_{\partial T})_F
	\Bigr]
\end{aligned}
\]
for all \(\boldsymbol\tau_h\in\Sigma_h^{-1}\), which uniquely determines
\(\dev\boldsymbol{\varepsilon}_w(\boldsymbol v_h,\boldsymbol\lambda_h)\).
For \(\boldsymbol v_h\in\mathbb V_h^{\div}\), integration by parts and
weak symmetry give, for all
\(\boldsymbol\tau_h\in\Sigma_h^{-1}\),
\[
\begin{aligned}
	(\dev\boldsymbol{\varepsilon}_w(\boldsymbol v_h,\boldsymbol\lambda_h),\boldsymbol\tau_h)
	={}&(\dev\boldsymbol{\varepsilon}_h(\boldsymbol v_h),\boldsymbol\tau_h)\\
	&-\sum_{T\in\mathcal T_h}\sum_{F\in\mathcal F(T)}
	\bigl(Q_F^t(\Pi_F(\boldsymbol v_h|_T))-\boldsymbol\lambda_h,
	\Pi_F\boldsymbol\tau_h\boldsymbol n_{\partial T}\bigr)_F.
\end{aligned}
\]
For \(\boldsymbol\tau_h\in\Sigma_h^{tn}\), the definition of \(b_h\) gives
\begin{equation}\label{eq:BrinkmanWeakStrainRelation}
	(\dev\boldsymbol{\varepsilon}_w(\boldsymbol v_h,\boldsymbol\lambda_h),\boldsymbol\tau_h)
	=-b_h(\boldsymbol\tau_h,0;\boldsymbol v_h).
\end{equation}

The stress-hybridized method seeks
\((\boldsymbol u_h,\boldsymbol\lambda_h,p_h)\in
\mathring{\mathbb V}_h^{\div}\times\Lambda_h^t\times
\mathring{\mathcal Q}_h\) such that
\begin{subequations}\label{eq:BrinkmanHybridized}
	\begin{align}
		\nu(\dev\boldsymbol{\varepsilon}_w(\boldsymbol u_h,\boldsymbol\lambda_h),
		\dev\boldsymbol{\varepsilon}_w(\boldsymbol v_h,\boldsymbol\mu_h))
		+
		(\boldsymbol u_h,\boldsymbol v_h)
		+
		(\div\boldsymbol v_h,p_h)
		&=
		(\boldsymbol f,\boldsymbol v_h),
		\label{eq:BrinkmanHybridized1}\\
		(\div\boldsymbol u_h,q_h)
		&=0,
		\label{eq:BrinkmanHybridized2}
	\end{align}
\end{subequations}
for all
\((\boldsymbol v_h,\boldsymbol\mu_h,q_h)\in
\mathring{\mathbb V}_h^{\div}\times\Lambda_h^t\times
\mathring{\mathcal Q}_h\).

\begin{theorem}\label{thm:BrinkmanHybridized}
	Method \eqref{eq:BrinkmanHybridized} is well posed and algebraically
	equivalent to \eqref{eq:BrinkmanDiscrete}. If
	\((\boldsymbol u_h,\boldsymbol\lambda_h,p_h)\) solves
	\eqref{eq:BrinkmanHybridized}, then
	$
	\boldsymbol\sigma_h
	:=\nu\dev\boldsymbol{\varepsilon}_w(\boldsymbol u_h,\boldsymbol\lambda_h)
	$
	together with \((\boldsymbol u_h,p_h)\) solves
	\eqref{eq:BrinkmanDiscrete}.
\end{theorem}
\begin{proof}
We first prove uniqueness. Let $\boldsymbol f=0$. By \eqref{eq:BrinkmanHybridized2},
$\div\boldsymbol u_h=0$. Taking
$\boldsymbol v_h=\boldsymbol u_h$ and
$\boldsymbol\mu_h=\boldsymbol\lambda_h$ in \eqref{eq:BrinkmanHybridized1} gives
$\dev\boldsymbol{\varepsilon}_w(\boldsymbol u_h,\boldsymbol\lambda_h)=0$ and $\boldsymbol u_h=0$.
The discrete divergence surjection then gives $p_h=0$. Hence
$\dev\boldsymbol{\varepsilon}_w(0,\boldsymbol\lambda_h)=0$, and therefore,
\[
\sum_{F\in\mathcal F(T)}
(\boldsymbol\lambda_h,
 \Pi_F\boldsymbol\tau\boldsymbol n_{\partial T})_F=0
\qquad
\forall\,\boldsymbol\tau\in
\Sigma_k^+(T;\mathbb S\cap\mathbb T),\; T\in\mathcal T_h.
\]
By local unisolvence, choose $\boldsymbol\tau\in\Sigma_k^+(T;\mathbb S\cap\mathbb T)$ with
\[
(\Pi_F\boldsymbol\tau\boldsymbol n_{\partial T},\boldsymbol q)_F
=
(\boldsymbol\lambda_h,\boldsymbol q)_F
\qquad
\forall\,\boldsymbol q\in\mathcal R_k^t(F),\;F\in\mathcal F(T),
\]
and with all cell DoFs equal to zero.
Taking \(\boldsymbol q=\boldsymbol\lambda_h|_F\) gives
\(\sum_{F\in\mathcal F(T)}\|\boldsymbol\lambda_h\|_{0,F}^2=0\);
hence \(\boldsymbol\lambda_h=0\). Thus the homogeneous problem
has only the zero solution, and finite dimensionality gives well-posedness.

For equivalence, set $\boldsymbol\sigma_h:=\nu\dev\boldsymbol{\varepsilon}_w(\boldsymbol u_h,\boldsymbol\lambda_h)$.
Testing \eqref{eq:BrinkmanHybridized1} with
	\(\boldsymbol v_h=\boldsymbol0\) and arbitrary
	\(\boldsymbol\mu_h\in\Lambda_h^t\), and noting that
	\([\![\Pi_F\boldsymbol\sigma_h\boldsymbol n_F]\!]
	\in\mathcal R_k^t(F)\), shows that the tangential--normal jumps of
	\(\boldsymbol\sigma_h\) vanish. Hence
	\(\boldsymbol\sigma_h\in\Sigma_h^{tn}\).
	Equation~\eqref{eq:BrinkmanWeakStrainRelation} and
	\eqref{eq:BrinkmanHybridized2} then give
	\eqref{eq:BrinkmanDiscrete1}, while taking
	\(\boldsymbol\mu_h=\boldsymbol0\) in \eqref{eq:BrinkmanHybridized1} gives
	\eqref{eq:BrinkmanDiscrete2}.
\end{proof}

\subsection{Stabilization-free virtual element method}
Adapting the local Stokes--Neumann virtual element construction of \cite{WeiHuangLi2021}, we introduce a stabilization-free virtual element realization of \eqref{eq:BrinkmanHybridized}. For
\(T\in\mathcal T_h\), define
\[
\begin{aligned}
	\mathbb V_{k+1}^{\mathrm{VE}}(T)
	:=\{\boldsymbol v\in H^1(T;\mathbb R^d):\;&
	\div\boldsymbol v\in\mathbb P_k(T),\ \text{there exists }s\in L^2(T)
	\text{ such that}\\
	&\div\boldsymbol{\varepsilon}(\boldsymbol v)+\grad s
	\in\mathbb P_{k-1}(T;\mathbb K)\boldsymbol x,\\
	&(\boldsymbol{\varepsilon}(\boldsymbol v)+s\boldsymbol I)\boldsymbol n_F
	\in\mathcal R_k^t(F)\oplus
	\mathbb P_{k+1}(F)\boldsymbol n_F
	\;\forall\,F\in\mathcal F(T)\}.
\end{aligned}
\]
The traction condition is understood in the weak normal-trace sense.
The DoFs are
\begin{subequations}\label{eq:BrinkmanVEDofs}
	\begin{align}
		(\Pi_F\boldsymbol v,\boldsymbol q)_F,
		&\qquad
		\boldsymbol q\in\mathcal R_k^t(F),\quad F\in\mathcal F(T),
		\label{eq:BrinkmanVEDofsTangential}\\
		(\boldsymbol v\cdot\boldsymbol n_F,q)_F,
		&\qquad
		q\in\mathbb P_{k+1}(F),\quad F\in\mathcal F(T),
		\label{eq:BrinkmanVEDofsNormal}\\
		(\boldsymbol v,\boldsymbol q)_T,
		&\qquad
		\boldsymbol q\in
		\grad\mathbb P_k(T)\oplus
		\mathbb P_{k-1}(T;\mathbb K)\boldsymbol x.
		\label{eq:BrinkmanVEDofsCell}
	\end{align}
\end{subequations}
Using the degree-\(k\) Koszul decomposition and noting that the traction
space contains the traces of \({\rm RM}(T)\), the local Stokes--Neumann
dimension and unisolvence argument of
\cite[Section~3.1]{WeiHuangLi2021} applies to the present data spaces.
Hence the DoFs \eqref{eq:BrinkmanVEDofs} are unisolvent for
\(\mathbb V_{k+1}^{\mathrm{VE}}(T)\). The same Koszul decomposition gives
\(\mathbb P_{k+1}(T;\mathbb R^d)
\subseteq\mathbb V_{k+1}^{\mathrm{VE}}(T)\), with \(s=0\) when \(k=0\).




For
\(\boldsymbol v\in\mathbb V_{k+1}^{\mathrm{VE}}(T)\), the DoFs
\eqref{eq:BrinkmanVEDofsNormal} and
\eqref{eq:BrinkmanVEDofsCell} determine the BDM reconstruction
\(I_{k+1,T}^{\div}\boldsymbol v\). Since
\(\div\boldsymbol v\in\mathbb P_k(T)\), the commuting property gives
\(\div I_{k+1,T}^{\div}\boldsymbol v=\div\boldsymbol v\). Moreover,
\begin{equation}\label{eq:BrinkmanVEBDMMoments}
	(I_{k+1,T}^{\div}\boldsymbol v-\boldsymbol v,\boldsymbol q)_T=0
	\qquad
	\forall\,\boldsymbol q\in\mathbb P_k(T;\mathbb R^d).
\end{equation}
Indeed, \eqref{eq:BrinkmanVEDofsCell} gives the identity on
\(\mathbb P_{k-1}(T;\mathbb K)\boldsymbol x\), while integration by parts,
the divergence identity, and \eqref{eq:BrinkmanVEDofsNormal} give it on
\(\grad\mathbb P_{k+1}(T)\). The degree-\(k\) Koszul decomposition then
proves \eqref{eq:BrinkmanVEBDMMoments}.

\begin{lemma}\label{lem:BrinkmanVEProjection}
	For every
	\(\boldsymbol v\in\mathbb V_{k+1}^{\mathrm{VE}}(T)\),
	the projection
	\(Q_T^\Sigma(\dev\grad\boldsymbol v)\) is computable from
	\eqref{eq:BrinkmanVEDofs}. Moreover,
	\begin{equation}\label{eq:BrinkmanVEPolynomialConsistency}
		Q_T^\Sigma(\dev\grad\boldsymbol p)
		=
		\dev\boldsymbol{\varepsilon}(\boldsymbol p)
		\qquad
		\forall\,\boldsymbol p\in
		\mathbb P_{k+1}(T;\mathbb R^d).
	\end{equation}
\end{lemma}

\begin{proof}
	For
	\(\boldsymbol\tau_h\in
	\Sigma_k^+(T;\mathbb S\cap\mathbb T)\),
	pointwise tracelessness and integration by parts give
	\begin{equation}\label{eq:BrinkmanVEProjectionFormula}
		\begin{aligned}
			(Q_T^\Sigma(\dev\grad\boldsymbol v),
			\boldsymbol\tau_h)_T
			={}&
			-(\boldsymbol v,\div\boldsymbol\tau_h)_T
			+\sum_{F\in\mathcal F(T)}
			(\Pi_F\boldsymbol v,
			\Pi_F\boldsymbol\tau_h\boldsymbol n_{\partial T})_F\\
			&+
			\sum_{F\in\mathcal F(T)}
			(\boldsymbol v\cdot\boldsymbol n_{\partial T},
			\boldsymbol n_{\partial T}^{\intercal}
			\boldsymbol\tau_h\boldsymbol n_{\partial T})_F .
		\end{aligned}
	\end{equation}
	The facet terms are given by
	\eqref{eq:BrinkmanVEDofsTangential} and
	\eqref{eq:BrinkmanVEDofsNormal}, while the volume term is computable from
	\eqref{eq:BrinkmanVEBDMMoments} since
	\(\div\boldsymbol\tau_h\in\mathbb P_k(T;\mathbb R^d)\).
	
	Since
	\(\dev\grad\boldsymbol p-\dev\boldsymbol{\varepsilon}(\boldsymbol p)
	=\skw\grad\boldsymbol p\in\mathbb P_k(T;\mathbb K)\), weak symmetry
	gives \eqref{eq:BrinkmanVEPolynomialConsistency}.
\end{proof}

Define the global virtual element space by
\[
\begin{aligned}
	\mathring{\mathbb V}_h^{\mathrm{VE}}
	:=\{\boldsymbol v_h\in L^2(\Omega;\mathbb R^d):
	&\,\boldsymbol v_h|_T\in\mathbb V_{k+1}^{\mathrm{VE}}(T)
	\quad\forall\,T\in\mathcal T_h; \;
	\text{DoFs \eqref{eq:BrinkmanVEDofsTangential}--%
		\eqref{eq:BrinkmanVEDofsNormal}}\\
	&\,\text{are single-valued across }\mathring{\mathcal F}_h
	\text{ and vanish on }\partial\Omega\}.
\end{aligned}
\]


The stabilization-free virtual element method seeks
\((\boldsymbol u_h,p_h)\in
\mathring{\mathbb V}_h^{\mathrm{VE}}
\times\mathring{\mathcal Q}_h\) such that
\begin{subequations}\label{eq:BrinkmanVE}
	\begin{align}
		\nu a_h^{\mathrm{VE}}(\boldsymbol u_h,\boldsymbol v_h)
		+(I_{k+1,h}^{\div}\boldsymbol u_h,
		I_{k+1,h}^{\div}\boldsymbol v_h)
		+(\div_h\boldsymbol v_h,p_h)
		&=
		(\boldsymbol f,I_{k+1,h}^{\div}\boldsymbol v_h),
		\label{eq:BrinkmanVE1}\\
		(\div_h\boldsymbol u_h,q_h)
		&=0,
		\label{eq:BrinkmanVE2}
	\end{align}
\end{subequations}
for all
\((\boldsymbol v_h,q_h)\in
\mathring{\mathbb V}_h^{\mathrm{VE}}
\times\mathring{\mathcal Q}_h\), where
\[
a_h^{\mathrm{VE}}(\boldsymbol w_h,\boldsymbol v_h)
:=
\bigl(
Q_h^\Sigma(\dev\grad_h\boldsymbol w_h),
Q_h^\Sigma(\dev\grad_h\boldsymbol v_h)
\bigr).
\]


\begin{theorem}\label{thm:BrinkmanVE}
	Method \eqref{eq:BrinkmanVE} is well posed and algebraically equivalent
	to \eqref{eq:BrinkmanHybridized}. If
	\((\boldsymbol u_h,p_h)\) solves \eqref{eq:BrinkmanVE}, then
	\(\boldsymbol\sigma_h
	:=\nu Q_h^\Sigma(\dev\grad_h\boldsymbol u_h)\), together with
	\((I_{k+1,h}^{\div}\boldsymbol u_h,p_h)\), solves
	\eqref{eq:BrinkmanDiscrete}.
\end{theorem}

\begin{proof}
	By the unisolvence of \eqref{eq:HdivDoF} and
	\eqref{eq:BrinkmanVEDofs}, the DoFs bijectively identify
	\(\boldsymbol v_h\in\mathring{\mathbb V}_h^{\mathrm{VE}}\) with
	\(
	\bigl(
	I_{k+1,h}^{\div}\boldsymbol v_h,\,
	Q_{\mathcal F_h}^t(\Pi_F\boldsymbol v_h)
	\bigr)
	\in
	\mathring{\mathbb V}_h^{\div}\times\Lambda_h^t,
	\)
where $\bigl(Q_{\mathcal F_h}^t(\Pi_F\boldsymbol v_h)\bigr)|_F:=Q_F^t((\Pi_F\boldsymbol v_h))$ for $F\in\mathcal{F}_h$.
	Equations~\eqref{eq:BrinkmanVEProjectionFormula} and
	\eqref{eq:BrinkmanVEBDMMoments} give
	\[
	Q_h^\Sigma(\dev\grad_h\boldsymbol v_h)
	=
	\dev\boldsymbol{\varepsilon}_w\bigl(
	I_{k+1,h}^{\div}\boldsymbol v_h,\,
	Q_{\mathcal F_h}^t(\Pi_F\boldsymbol v_h)
	\bigr).
	\]
	Together with the commuting property, this shows that
	\eqref{eq:BrinkmanVE} coincides with
	\eqref{eq:BrinkmanHybridized} under the above correspondence.
	The result follows from Theorem~\ref{thm:BrinkmanHybridized}.
\end{proof}

\section{Numerical Examples}\label{sec:BrinkmanNumericalResults}

We test optimal convergence, Darcy-limit robustness, and obstacle through-flow with slip boundary conditions on uniform simplicial meshes using iFEM~\cite{chen2009ifem}.

\subsection{Error Estimates for a Smooth Exact Solution}

\begin{example}\label{exampleBrinkmanSmooth}
	\normalfont
	Let \(\Omega=(0,1)^3\),
	\(\psi_3=x^2(x-1)^2y^2(y-1)^2z^2(z-1)^2\), and take
	\[
	\boldsymbol u=\curl(\psi_3,\psi_3,\psi_3)^{\intercal},
	\qquad p=-x^5-y^5-z^5+\frac12.
	\]
	Then \(\boldsymbol\sigma=\nu\boldsymbol{\varepsilon}(\boldsymbol u)\),
	and \(\boldsymbol f\) is determined by \eqref{eq:Brinkman_intro}.
\end{example}

Define
\begin{equation}\label{eq:BrinkmanNumericalViscosityError}
	E_{\nu,h}
	:=
	\nu^{-1/2}\|\boldsymbol\sigma-\boldsymbol\sigma_h\|_{0,h}
	+\nu^{1/2}|\boldsymbol u-\boldsymbol u_h|_{1,h}
	+\|\boldsymbol u-\boldsymbol u_h\|_0+\|p-p_h\|_0.
\end{equation}
Table~\ref{tab:Brinkman-smooth-natural} reports \(E_{\nu,h}\) and its
observed order for \(k=0,1\) and four viscosities.

\begin{table}[htbp]
	\centering
	\caption{Numerical results for
		Example~\ref{exampleBrinkmanSmooth}.}
	\label{tab:Brinkman-smooth-natural}
	\footnotesize
	\setlength{\tabcolsep}{2pt}
	\renewcommand{\arraystretch}{0.92}
	\begin{tabular*}{\textwidth}{@{\extracolsep{\fill}}c c cccc cccc@{}}
		\toprule
		\multirow{2}{*}{\(k\)} &
		\multirow{2}{*}{\(h\)}
		& \multicolumn{2}{c}{\(\nu=1\)}
		& \multicolumn{2}{c}{\(\nu=10^{-2}\)}
		& \multicolumn{2}{c}{\(\nu=10^{-4}\)}
		& \multicolumn{2}{c}{\(\nu=10^{-6}\)} \\
		\cmidrule(lr){3-4}
		\cmidrule(lr){5-6}
		\cmidrule(lr){7-8}
		\cmidrule(lr){9-10}
		& & \(E_{\nu,h}\) & order
		& \(E_{\nu,h}\) & order
		& \(E_{\nu,h}\) & order
		& \(E_{\nu,h}\) & order \\
		\midrule
		\multirow{5}{*}{\(0\)}
		& \(2^{-1}\) & 3.033e-01 & -- & 2.957e-01 & -- & 2.949e-01 & -- & 2.948e-01 & -- \\
		& \(2^{-2}\) & 1.700e-01 & 0.84 & 1.655e-01 & 0.84 & 1.650e-01 & 0.84 & 1.649e-01 & 0.84 \\
		& \(2^{-3}\) & 8.773e-02 & 0.95 & 8.530e-02 & 0.96 & 8.505e-02 & 0.96 & 8.503e-02 & 0.96 \\
		& \(2^{-4}\) & 4.422e-02 & 0.99 & 4.298e-02 & 0.99 & 4.286e-02 & 0.99 & 4.284e-02 & 0.99 \\
		& \(2^{-5}\) & 2.215e-02 & 1.00 & 2.153e-02 & 1.00 & 2.147e-02 & 1.00 & 2.146e-02 & 1.00 \\
		\midrule
		\multirow{4}{*}{\(1\)}
		& \(2^{-1}\) & 7.965e-02 & -- & 7.754e-02 & -- & 7.706e-02 & -- & 7.689e-02 & -- \\
		& \(2^{-2}\) & 2.262e-02 & 1.82 & 2.136e-02 & 1.86 & 2.125e-02 & 1.86 & 2.122e-02 & 1.86 \\
		& \(2^{-3}\) & 5.849e-03 & 1.95 & 5.477e-03 & 1.96 & 5.441e-03 & 1.97 & 5.437e-03 & 1.96 \\
		& \(2^{-4}\) & 1.475e-03 & 1.99 & 1.378e-03 & 1.99 & 1.368e-03 & 1.99 & 1.368e-03 & 1.99 \\
		\bottomrule
	\end{tabular*}
\end{table}

Across both polynomial degrees and all tested viscosities, the weighted
errors converge with order \(k+1\), confirming
Corollary~\ref{cor:BrinkmanExactError}.

\subsection{Darcy Limit and Boundary Layer}

\begin{example}
	\normalfont
	Let \(\Omega=(0,1)^2\) and \(g(y)=y^2(1-y)^2\). We take \(k=1\) and
	\[
	\boldsymbol u^0=(-x(1-x)g',(1-2x)g)^{\intercal},
	\qquad p^0=x^2+y^2-\frac23,
	\qquad \boldsymbol f=\boldsymbol u^0-\nabla p^0.
	\]
	This Darcy pair satisfies \(\boldsymbol u^0\cdot\boldsymbol n=0\), while
	its tangential trace is nonzero on \(x=0,1\). We solve the discrete
	Brinkman problem with the same right-hand side for each \(\nu\).
\end{example}

For the quantities reported in
Figure~\ref{fig:Brinkman-boundary-layer-diagnostics}, define
\[
	\begin{aligned}
	E^0_{\nu,h}
	&:=\nu^{1/2}
	\bigl\|Q_h^\Sigma(\dev\grad\boldsymbol u^0)
	-\nu^{-1}\boldsymbol\sigma_h\bigr\|_0
	+\|\boldsymbol u^0-\boldsymbol u_h\|_0+\|p^0-p_h\|_0,\\
	R_{\nu,h}
	&:=\frac{E^0_{\nu,h}}
	{\nu^{1/4}\|\boldsymbol f\|_{H(\curl)}
		+h^2(|\boldsymbol u^0|_2+|p^0|_2)}.
\end{aligned}
\]
Figure~\ref{fig:Brinkman-boundary-layer-diagnostics}(a) shows that
\(R_{\nu,h}\) remains uniformly bounded over the tested \(h\) and \(\nu\),
confirming \eqref{eq:BrinkmanDarcyLimitError}. Panel~(b) illustrates the
increasing boundary-layer concentration of
\(|\nu^{-1}\boldsymbol\sigma_h|_{\rm F}\) near \(x=0,1\).

\begin{figure}[htbp]
	\centering
	\includegraphics[width=0.94\textwidth]
	{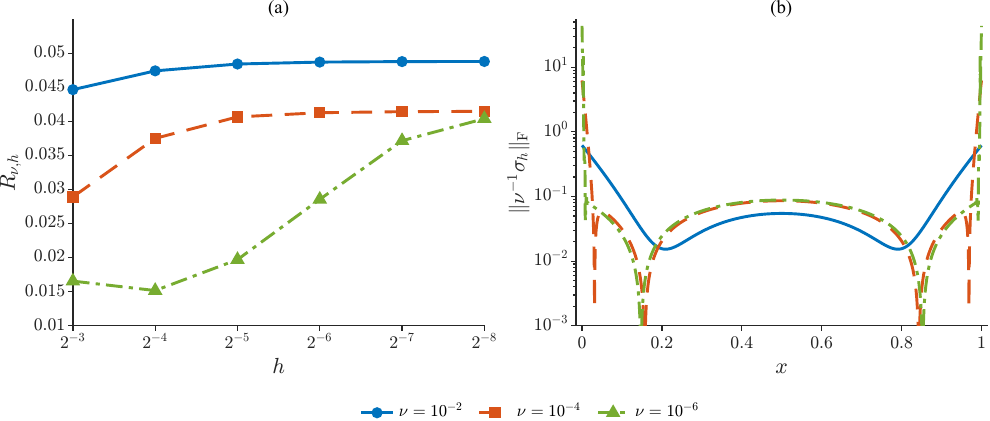}
	\caption{Parameter-uniform boundary-layer diagnostics for \(d=2\) and
		\(k=1\): (a) normalized errors \(R_{\nu,h}\); (b) averaged one-sided
		traces of the Frobenius norm \(|\nu^{-1}\boldsymbol\sigma_h|_{\rm F}\)
		along \(y=1/2\) on the finest mesh \(h=2^{-8}\).}
	\label{fig:Brinkman-boundary-layer-diagnostics}
\end{figure}

\subsection{Obstacle Through-Flow}

\begin{example}
	\normalfont
	We adopt the channel--cylinder geometry of the classical
	Navier--Stokes benchmark in \cite{SchaeferTurek1996}.
	Let \(H=0.41\), and let \(D_{\rm obs}\) be the disk centered at
	\((0.2,0.2)\) with radius \(0.05\). Define
	\(\Omega:=(0,2.2)\times(0,H)\setminus\overline{D_{\rm obs}}\).
	We take \(k=1\) and \(\boldsymbol f=\boldsymbol0\), and prescribe
	\[
	\boldsymbol u\cdot\boldsymbol n=
	\begin{cases}
		-4y(H-y)/H^2, & \text{on }x=0,\\
		\phantom{-}4y(H-y)/H^2, & \text{on }x=2.2.
	\end{cases}
	\]
	On the channel walls and the obstacle boundary, we set
	\(\boldsymbol u\cdot\boldsymbol n=0\), and the tangential traction vanishes
	on \(\partial\Omega\).
	
	We use uniform refinements of a fixed polygonal approximation
	\(\Omega_{\rm poly}\) of \(\Omega\).
	For \(k=1\), the prescribed normal trace is imposed exactly in
	an affine translate of the BDM trial space, while the homogeneous space in
	Remark~\ref{rem:slip_discrete} is used for testing. Let \(\Gamma_{\rm in}\),
	\(\Gamma_{\rm out}\), and \(\Gamma_{\rm obs}\) denote the inlet, outlet,
	and polygonal obstacle boundary of \(\Omega_{\rm poly}\), respectively. Set
	\[
	\begin{aligned}
		D^h&=-\int_{\Gamma_{\rm obs}}
		\bigl[(\boldsymbol\sigma_h+p_h\boldsymbol I)\boldsymbol n\bigr]_1\,\dd s,
		&\Delta p_h&=\frac{1}{|\Gamma_{\rm in}|}
		\int_{\Gamma_{\rm in}}p_h\,\dd s
		-\frac{1}{|\Gamma_{\rm out}|}
		\int_{\Gamma_{\rm out}}p_h\,\dd s.
	\end{aligned}
	\]
	The signs of \(D^h\) and \(\Delta p_h\) follow the fluid outward normal and
	the pressure convention in \eqref{eq:Brinkman_intro}, respectively.
\end{example}

Figure~\ref{fig:Brinkman-obstacle-throughflow} shows a localized velocity
disturbance near the obstacle and an essentially parallel downstream flow,
consistent with Brinkman screening~\cite{DurlofskyBrady1987}. The nonzero
wall-parallel velocity reflects the slip condition. Under the convention
in \eqref{eq:Brinkman_intro}, the pressure increases downstream, consistent
with \(\Delta p_h<0\).

\begin{figure}[htbp]
	\centering
	\includegraphics[width=\textwidth]
	{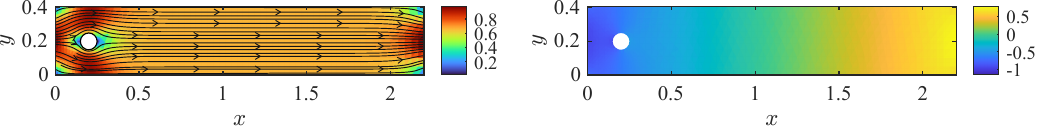}
	\caption{Velocity magnitude and streamlines (left) and pressure (right)
		for \(\nu=10^{-2}\), \(k=1\), and mesh level~4 on
		\(\Omega_{\rm poly}\).}
	\label{fig:Brinkman-obstacle-throughflow}
\end{figure}

\begin{table}[htbp]
	\centering
	\caption{Obstacle drag \(D^h\), pressure difference \(\Delta p_h\), and
		the divergence diagnostic for \(k=1\) under uniform refinement. Here
		\(R_{\div}^h:=\max_\nu\|\div\boldsymbol u_h\|_0\), where the maximum
		is taken over the four displayed viscosities.}
	\label{tab:Brinkman-obstacle-outputs}
	\scriptsize
	\setlength{\tabcolsep}{2pt}
	\renewcommand{\arraystretch}{1.2}
	\begin{tabular}{c c cc cc cc cc c}
		\toprule
		\multirow{2}{*}{level} & \multirow{2}{*}{DoFs}
		& \multicolumn{2}{c}{\(\nu=1\)}
		& \multicolumn{2}{c}{\(\nu=10^{-2}\)}
		& \multicolumn{2}{c}{\(\nu=10^{-4}\)}
		& \multicolumn{2}{c}{\(\nu=10^{-6}\)}
		& \multicolumn{1}{c}{divergence} \\
		\cmidrule(lr){3-4}
		\cmidrule(lr){5-6}
		\cmidrule(lr){7-8}
		\cmidrule(lr){9-10}
		\cmidrule(lr){11-11}
		& & \(D^h\) & \(\Delta p_h\) & \(D^h\) & \(\Delta p_h\)
		& \(D^h\) & \(\Delta p_h\) & \(D^h\) & \(\Delta p_h\)
		& \(R_{\div}^h\) \\
		\midrule
		1 & 1031 & 3.624 & -11.113 & 0.052 & -1.605 & 0.012 & -1.500 & 0.011 & -1.498 & 1.979e-14 \\
		2 & 4307 & 3.784 & -11.270 & 0.054 & -1.605 & 0.012 & -1.497 & 0.011 & -1.495 & 7.825e-14 \\
		3 & 17591 & 3.891 & -11.263 & 0.055 & -1.605 & 0.012 & -1.496 & 0.011 & -1.494 & 7.157e-14 \\
		4 & 71087 & 3.986 & -11.356 & 0.057 & -1.606 & 0.012 & -1.496 & 0.011 & -1.494 & 2.368e-13 \\
		\bottomrule
	\end{tabular}
\end{table}
The values of \(R_{\div}^h\) in
Table~\ref{tab:Brinkman-obstacle-outputs} remain at roundoff level,
numerically confirming exact discrete incompressibility. Since
\(\boldsymbol f=\boldsymbol0\), \(|\Delta p_h|\) is the pressure drop
required to sustain the prescribed through-flow against porous and viscous
resistance, while \(D^h\) measures the horizontal force exerted on the
obstacle. At \(\nu=1\), viscous effects are substantial, consistent with the larger values of both \(D^h\) and \(|\Delta p_h|\). As \(\nu\) decreases, both quantities
approach limiting values characteristic of the Darcy-dominated regime
\cite{Brinkman1949,NieldBejan2017}. Under refinement, \(\Delta p_h\)
stabilizes more rapidly than \(D^h\). This difference reflects the fact
that \(\Delta p_h\) is obtained from boundary-averaged pressures, whereas
\(D^h\) is a traction functional on the obstacle boundary and is therefore
more sensitive to the local mesh resolution near the obstacle.

\bibliographystyle{siamplain}
\bibliography{references}

\end{document}